\documentclass{amsart}

\usepackage{etoolbox}
\patchcmd{\section}{\scshape}{\bfseries}{}{}
\patchcmd{\subsection}{-.5em}{.5em}{}{}
\makeatletter
\renewcommand{\@secnumfont}{\bfseries}
\makeatother

\usepackage{amssymb}
\usepackage{enumerate}
\usepackage{url} 
\usepackage{flagderiv}
\usepackage{cmll}

\usepackage[top=3cm, left=2.3cm, bottom=4.5cm, right=2.2cm]{geometry}

\usepackage{booktabs}

\usepackage{xcolor}

\newtheorem{theorem}{Theorem}[section]

\newtheorem{lemma}[theorem]{Lemma}

\newtheorem{definition}[theorem]{Definition}

\newtheorem{example}[theorem]{Example}

\begin{document}
\mathchardef\mhyphen="2D

\title[Formalizing Relations in Type Theory]{Formalizing Relations in Type Theory}

\author[F. Kachapova]{Farida Kachapova}
\address{Department of Mathematical Sciences
\\Auckland University of Technology
\\New Zealand}

\email{farida.kachapova@aut.ac.nz}

\date{} 

\begin{abstract}
 Type theory plays an important role in foundations of mathematics as a framework for formalizing mathematics and a base for proof assistants providing semi-automatic proof checking and construction. Derivation of each theorem in type theory results in a formal term encapsulating the whole proof process. In this paper we use a variant of type theory, namely the Calculus of Constructions with Definitions, to formalize the standard theory of binary relations. This includes basic operations on relations, criteria for special properties of relations, invariance of these properties under the basic operations, equivalence relation, well-ordering, and transfinite induction. Definitions and proofs are presented as flag-style derivations. 
\end{abstract}

\subjclass[2020]{Primary 03B30; Secondary 03B38}

\keywords{Type theory, calculus of constructions, binary relation, transfinite induction, flag-style derivation}

\maketitle

\section{Introduction}

First type theories were proposed by B. Russell \cite{Russ96} as a foundation of mathematics. Other important type theories are typed $\lambda$-calculus introduced by A. Church \cite{Chur40} and intuitionistic type theory introduced by P. Martin-L{\"o}f \cite{Mart85}. A higher-order typed $\lambda$-calculus known as Calculus of Constructions (CoC) was created by T. Coquand \cite{Coq88}. Variants of CoC make formal bases of proof assistants, which are computer tools for formalizing and developing mathematics. In particular, the well-known proof assistant Coq is based on the strong variant of CoC called the Calculus of Inductive Constructions (CIC).

Here we use the variant $\lambda D$ of CoC developed in \cite{Ned14}; $\lambda D$ is called the Calculus of Constructions with Definitions. We choose $\lambda D$ because of its following useful properties.
\begin{itemize}
\item[--] In $\lambda D$, as in other variants of CoC, proofs are expressed as formal terms and thus are incorporated in the system.
\item[--] In $\lambda D$ type checking is decidable and therefore proof checking is decidable. So the correctness of a proof can be checked by an algorithm.
\item[--] $\lambda D$ is strongly normalizing, which implies the logical consistency of this theory, even with classical logic (when no extra axioms are added) - see \cite{Baren92}.
\end{itemize}

The theory $\lambda D$ is weaker than CIC because $\lambda D$ does not have inductive types. This does not limit its capability for formalizing mathematics because in $\lambda D$ we can use axiomatic approach and higher-order logic to express the objects that CIC defines with inductive types. 

In Section 2 we briefly describe the theory $\lambda D$, derived rules of intuitionistic logic in $\lambda D$, and the classical axiom of excluded third that can be added to $\lambda D$ if necessary; we also briefly explain the flag format derivation.
In Section 3 we describe the equality in $\lambda D$ and its derived properties. 

In Section 4 we study binary relations in $\lambda D$, operations on relations, and their properties. In Section 5 we formally prove criteria of relexivity, symmetry, antisymmetry and transitivity, and study the invariance of these properties under some basic operations. In Section 6 we formally define partitions in $\lambda D$ and provide a proof of their correspondence with equivalence relations. In Section 6 we also provide an example of partial order with  a formal proof, definition of well-ordering in $\lambda D$ and a formal proof of the principle of transfinite induction.  

In our formalizations we aim to keep the language and theorems as close as possible to the ones of standard mathematics.
In definitions and proofs we use the flag-style derivation described in \cite{Ned14}. Long formal derivations are moved from the main text to Appendices for better readability. 

\section{Type Theory $\lambda D$}
In \cite{Ned14} Nederpelt and Geuvers developed a formal theory  $\lambda D$ and formalized some parts of logic and mathematics in it. Here we briefly describe main features of $\lambda D$.

\subsection{Type Theory $\lambda D$}

The language of $\lambda D$ described in \cite{Ned14} has an infinite set of variables, $V$, and an infinite set of constants, $C$; these two sets are disjoint. There are also special symbols $\square$ and $*$. 

\begin{definition}
Expressions of the language are defined recursively as follows.
\begin{enumerate}
\item Each variable is an expression.

\item Each constant is an expression.

\item Constant * is an expression.

\item Constant $\square$ is an expression.

\item (Application) If $A$ and $B$ are expressions, then $AB$ is an expression.

\item (Abstraction) If $A$,  $B$ are expressions and $x$ is a variable, then $\lambda x:A.B$ is an expression.

\item (Dependent Product) If $A$,  $B$ are expressions and $x$ is a variable, then $\Pi x:A.B$ is an expression.

\item If $A_1,A_2,\ldots,A_n$ are expressions and $c$ is a constant, then $c\left(A_1,A_2,\ldots,A_n\right)$ is an expression.
\end{enumerate}
\end{definition}

An expression $A\rightarrow B$ is introduced as a particular type of Dependent Product from (7) when $x$ is not a free variable in $B$.

\begin{definition}

\begin{enumerate}
\item A \textbf{statement} is of the form $M:N$, where $M$ and $N$ are expressions.

\item A \textbf{declaration} is of the form $x:N$, where $x$ is a variable and $N$ is an expression.

\item A \textbf{descriptive definition} is of the form:
\[\bar{x}:\bar{A}\rhd c(\bar{x}):=M:N,\]
where $\bar{x}$ is a list $x_1,x_2,\ldots,x_n$ of variables, $\bar{A}$ is a list $A_1,A_2,\ldots,A_n$ of expressions, $c$ is a constant, and $M$ and $N$ are expressions.

\item A \textbf{primitive definition} is of the form:
\[\bar{x}:\bar{A}\rhd c(\bar{x}):=\Bot:N,\]
where $\bar{x}$, $\bar{A}$, and $c$ are described the same way as in (3), and $N$ is an expression. The symbol $\Bot$  denotes the non-existing definiens. Primitive definitions are used for introducing axioms where no proof terms are needed.

\item A \textbf{definition} is a descriptive definition or a primitive definition.

\item A \textbf{judgement} is of the form:
\[\Delta;\Gamma \vdash M:N,\]
where $M$ and $N$ are expressions of the language, $\Delta$ is an \textbf{environment} (a properly constructed sequence of definitions) and $\Gamma$ is a \textbf{context} (a properly constructed sequence of declarations).
\end{enumerate}
\end{definition}

For brevity we often use implicit variables in definitions, that is we omit the previously declared variables $\bar{x}$ in $c(\bar{x})$ in (3) and (4).

The following informally explains the meaning of expressions.
\begin{enumerate}
\item If an expression $M$ appears in a derived statement of the form $M$ $:*$, then $M$ is interpreted as a \textbf{type}, which represents a set or a proposition. 

\textit{Note}: There is only one type $*$ in $\lambda D$. But informally we often use $*_p$ for propositions and $*_s$ for sets to make proofs more readable.

\item If an expression $M$ appears in a derived statement of the form $M:N$, where $N$ is a type, then $M$ is interpreted as an object at the lowest level. 

When $N$ is interpreted as a set, then $M$ is regarded as an element of this set. 

When $N$ is interpreted as a proposition, then $M$ is regarded as a proof (or a proof term) of this proposition.

\item The symbol $\square$ represents the highest level.

\item \textbf{Sort} is $*$ or $\square$. Letters $s, s_1,s_2,\ldots$ are used as variables for sorts.

\item  If an expression $M$ appears in a statement of the form $M:\square$, then $M$ is called a \textbf{kind}. 
$\lambda D$ contains the derivation rule:
\[\varnothing;\varnothing \vdash *:\square,\]
which is its (only) axiom because it has an empty environment and an empty context.
\end{enumerate}

Further details of the language and derivation rules of the theory $\lambda D$ can be found in \cite{Ned14}. Judgments are formally derived in $\lambda D$ using the derivation rules. 
\medskip

\subsection{Flag Format of Derivations}

The flag-style deduction was introduced by Ja\'{s}kowski  \cite{Jas67} and Fitch \cite{Fitch52}. A derivation in the flag format is a linear deduction. Each "flag" (a rectangular box) contains a declaration that introduces a variable or an assumption; a collection of already introduced variables and assumptions makes the current context. The scope of the variable or assumption is established by the "flag pole". In the scope we construct definitions and proof terms for proving statements{/} theorems in $\lambda D$. Each new flag extends the context and at the end of each flag pole the context is reduced by the corresponding declaration.
For brevity we can combine several declarations in one flag.

More details on the flag-style deduction can be found in \cite{Ned11} and \cite{Ned14}.
\medskip

\subsection{Logic in $\lambda D$\label{subs_logic}}

The rules of intuitionistic logic are derived in the theory $\lambda D$ as shown in \cite{Ned14}. We briefly describe it here by showing the introduction and elimination rules for logical connectives and quantifiers. 

\subsubsection{Implication}
The logical implication $A\Rightarrow B$ is identified with the arrow type $A\rightarrow B$. The rules for implication follow from the following general rules for the arrow type (we write them in the flag format):

\setlength{\derivskip}{4pt}
\begin{flagderiv}
\introduce*{}{A:s_1\;|\;B:s_2}{}
\step*{}{A\rightarrow B:s_2}{}
\assume*{}{u:A\rightarrow B\;|\;v:A}{}
\step*{}{uv:B}{}
\done
\introduce*{}{x:A}{}
\skipsteps*{\dots}{}
\step*{}{M:B}{}
\conclude*{}{\lambda x:A.M\;:\;A\rightarrow B}{}
\end{flagderiv}

Here $x$ is not a free variable in $B$.

In $\lambda D$ arrows are right associative, that is $A\rightarrow B\rightarrow C$ is a shorthand for $A\rightarrow (B\rightarrow C)$.

\subsubsection{Falsity and Negation}
Falsity $\bot$ is introduced in $\lambda D$ by:
\[\bot:=\Pi A:*_p.A\;:\;*_p.\]

From this definition we get a rule for falsity:

\newpage
\begin{flagderiv}
\introduce*{}{B:*_p}{}
\skipsteps*{\dots}{}
\step*{}{u:\bot}{}
\step*{}{u:\Pi A:*_p.A}{}
\step*{}{uB:B}{}
\end{flagderiv}

The rule states that falsity implies any proposition.

As usual, negation is defined by: $\neg A:= A\rightarrow \bot$. 

Other logical connectives and quantifiers are also defined using second order encoding. Here we only list their derived rules and names of the corresponding terms, without details of their construction. 
The exact values of the terms can be found in \cite{Ned14}. 

Some of our flag derivations contain the proof terms that will be re-used in other proofs; such proof terms are written in bold font, e.g. $\boldsymbol{\wedge}\textbf{-in}$ in the first derived rule for conjunction as follows.

\subsubsection{Conjunction}
These are derived rules for conjunction $\wedge$: 

\begin{flagderiv}
\introduce*{}{A,B:*_p}{}
\assume*{}{u:A\;|\;v:B}{}
\step*{}{\boldsymbol{\wedge}\textbf{-in}(A,B,u,v)  \;:\;A\wedge B}{}
\done
\assume*{}{w:A\wedge B}{}
\step*{}{\boldsymbol{\wedge}\textbf{-el}_1(A,B,w)  \;:\;A}{}
\step*{}{\boldsymbol{\wedge}\textbf{-el}_2(A,B,w)  \;:\;B}{}
\end{flagderiv}

\subsubsection{Disjunction}

These are derived rules for disjunction $\vee$:
\begin{flagderiv}
\introduce*{}{A,B:*_p}{}
\assume*{}{u:A}{}
\step*{}{\boldsymbol{\vee}\textbf{-in}_1(A,B,u)  \;:\;A\vee B}{}
\done
\assume*{}{u:B}{}
\step*{}{\boldsymbol{\vee}\textbf{-in}_2(A,B,u)  \;:\;A\vee B}{}
\done
\assume*{}{C:*_p}{}
\assume*{}{u:A\vee B\;|\;v:A\Rightarrow C\;|\;w:B\Rightarrow C}{}
\step*{}{\boldsymbol{\vee}\textbf{-el}(A,B,C,u,v,w)  \;:\;C}{}
\end{flagderiv}

\subsubsection{Bi-implication}

Bi-implication $\Leftrightarrow$ has the standard definition: 
\[(A\Leftrightarrow B):=(A\Rightarrow B)\wedge (B\Rightarrow A).\]

\begin{lemma}
We will often use this lemma to prove bi-implication $A\Leftrightarrow B$.

\begin{flagderiv}
\introduce*{}{A,B:*_p}{}
\assume*{}{u:A\Rightarrow B\;|\;v:B\Rightarrow A}{}
\step*{}{\textbf{bi-impl}(A,B,u,v):=\wedge\text{-in}(A\Rightarrow B,B\Rightarrow A,u,v)\;:\;A\Leftrightarrow B}{}
\end{flagderiv}
\label{lemma_bi-impl}
\end{lemma}

\subsubsection{Universal Quantifier}
The universal quantifier $\forall$ is defined through the dependent product:

\begin{flagderiv}
\introduce*{}{S:*_s\;|\;P:S\rightarrow *_p}{}
\step*{}{\text{Definition  }\forall (S,P)\;:=\Pi x:S.Px\;:\;*_p}{}
\step*{} {\text{Notation}: \mathbf{(\forall x:S.Px)} \text{ for }\forall (S,P)}{}
\end{flagderiv}

\subsubsection{Existential Quantifier}

These are derived rules for the existential quantifier $\exists$.
\vspace{-0.2cm}
\begin{flagderiv}
\introduce*{}{S:*_s\;|\;P:S\rightarrow *_p}{}
\introduce*{}{y:S\;|\;u:Py}{}
\step*{}{\boldsymbol{\exists}\textbf{-in}(S,P,y,u)  \;:\;(\exists x:S.Px)}{}
\done
\assume*{}{C:*_p}{}
\assume*{}{u:(\exists x:S.Px)\;|\;v:(\forall x:S.(Px\Rightarrow C))}{}
\step*{}{\boldsymbol{\exists}\textbf{-el}(S,P,u,C,v)  \;:\;C}{}
\end{flagderiv}

Here $x$ is not a free variable in $C$. 

\subsubsection{Classical Logic}
We use mostly intuitionistic logic. But sometimes classical logic is needed; in these cases we add the following \textbf{Axiom of Excluded Third}:
\vspace{-0.2cm}
\begin{flagderiv}
\introduce*{}{A:*_p}{}
\step*{}{\textbf{exc-thrd}(A):=\Bot \;:\;A\vee\neg A}{}
\end{flagderiv}

This axiom implies the \textbf{Double Negation} theorem:
\vspace{-0.2cm}
\begin{flagderiv}
\introduce*{}{A:*_p}{}
\step*{}{\textbf{doub-neg}(A):(\neg\neg A\Rightarrow A)}{}
\end{flagderiv}

\subsection{Sets in $\lambda D$}
Here we briefly repeat some definitions from \cite{Ned14}
relating to sets, in particular, subsets of type $S$.

\begin{flagderiv}
\introduce*{}{S:*_s}{}
\step*{}{\boldsymbol{ps(S)}:=S\rightarrow *_p}{\textbf{Power set of S}}
\introduce*{}{V:ps(S)}{}
\step*{}{\text{Notation: }
\boldsymbol{\{x:S\;|\;x\varepsilon V\}}\text{ for }\lambda x:S.Vx
}{}
\introduce*{}{x: S}{}
\step*{}{\boldsymbol{element}(S,x,V):=Vx:*_p}{}
\step*{}{\text{Notation: }x\varepsilon_S V \text{ or }x\varepsilon V\text{ for }element(S,x,V)}{}
\end{flagderiv}

Thus, a subset $V$ of $S$ is regarded as a predicate on $S$ and $x\varepsilon V$ means $x$ satisfies the predicate $V$.

\section{Intensional Equality in $\lambda D$ }
\label{sect_int_equality}

Here we introduce intensional equality for elements of any type; we will call it just equality. In the next section we will introduce extensional equality and the axiom of extensionality relating the two types of equality.

\begin{flagderiv}
\introduce*{}{S:*}{}
\introduce*{}{x,y: S}{}
\step*{}{eq(S,x,y):=\Pi P:S\rightarrow *_p. (Px\Rightarrow Py):*_p}{}
\step*{} {\text{Notation}: \boldsymbol{x=_Sy} \text{ for } eq(S,x,y)}{\textbf{Intensional equality}}
\end{flagderiv}

\subsection{Properties of Equality}

\subsubsection{Reflexivity}
The following diagram proves the reflexivity property of equality in $\lambda D$.

\vspace{-0.2cm}
\begin{flagderiv}
\introduce*{}{S:*\;|\;x: S}{}
\introduce*{}{P: S\rightarrow *_p}{}
\step*{}{Px:*_p}{}
\step*{}{a_1:=\lambda u:Px.u:Px\Rightarrow Px}{}
\conclude*{}{eq\mhyphen refl(S,x)=\lambda P:S\rightarrow *_p.a_1:(\Pi P:S\rightarrow *_p.(Px\Rightarrow Px))}{}
\step*{}{\boldsymbol{eq\mhyphen refl}(S,x):x=_Sx}
{}
\end{flagderiv}

Proof terms are constructed similarly for the following properties of Substitutivity, Congruence, Symmetry, and Transitivity (see \cite{Ned14}).

\subsubsection{Substitutivity}

Substitutivity means that equality is consistent with predicates of corresponding types.

\vspace{-0.2cm}
\begin{flagderiv}
\introduce*{}{S:*}{}
\introduce*{}{P: S\rightarrow *_p}{}
\introduce*{}{x,y: S\;|\;u:x=_Sy\;|\;v:Px}{}
\step*{}{\boldsymbol{eq\mhyphen subs}
(S,P,x,y,u,v):Py}{}
\end{flagderiv}

\subsubsection{Congruence}

Congruence means that equality is consistent with functions of corresponding types.
\vspace{-0.2cm}

\begin{flagderiv}
\introduce*{}{Q,S:*}{}
\introduce*{}{f: Q\rightarrow S}{}
\introduce*{}{x,y: Q\;|\;u:x=_Qy}{}
\step*{}{\boldsymbol{eq\mhyphen cong}
(Q,S,f,x,y,u):fx=_Sfy}{}
\end{flagderiv}

\subsubsection{Symmetry}
The following diagram expresses the symmetry property of equality in $\lambda D$.

\begin{flagderiv}
\introduce*{}{S:*}{}
\introduce*{}{x,y: S\;|\;u:x=_Sy}{}
\step*{}{\boldsymbol{eq\mhyphen sym
}(S,x,y,u):y=_Sx}{}
\end{flagderiv}

\subsubsection{Transitivity}
The following diagram expresses the transitivity property of equality in $\lambda D$.

\begin{flagderiv}
\introduce*{}{S:*}{}
\introduce*{}{x,y,z: S\;|\;u:x=_Sy\;|\;v:y=_Sz}{}
\step*{}{\boldsymbol{eq\mhyphen trans}
(S,x,y,z,u,v):x=_Sz}{}
\end{flagderiv}

\section{Relations in Type Theory}

\subsection{Sets in $\lambda D$}
Here we briefly repeat some definitions from \cite{Ned14}
relating to sets, in particular, subsets of type $S$.

\begin{flagderiv}
\introduce*{}{S:*_s}{}
\step*{}{\boldsymbol{ps(S)}:=S\rightarrow *_p}{\textbf{Power set of S}}
\introduce*{}{V:ps(S)}{}
\step*{}{\text{Notation: }
\boldsymbol{\{x:S\;|\;x\varepsilon V\}}\text{ for }\lambda x:S.Vx
}{}
\introduce*{}{x: S}{}
\step*{}{\boldsymbol{element}(S,x,V):=Vx:*_p}{}
\step*{}{\text{Notation: }x\varepsilon_S V \text{ or }x\varepsilon V\text{ for }element(S,x,V)}{}
\end{flagderiv}

Thus, a subset $V$ of $S$ is regarded as a predicate on $S$ and $x\varepsilon V$ means $x$ satisfies the predicate $V$.

\subsection{Defining Binary Relations in $\lambda D$}
Binary relations are introduced in \cite{Ned14}, together with the properties of reflexivity, symmetry, antisymmetry, and transitivity, and definitions of equivalence relation and partial order. We use them as a starting point for formalizing the theory of binary relations in $\lambda D$.

A relation on $S$ is a binary predicate on $S$, which is regarded in $\lambda D$ as a composition of unary predicates. For brevity we introduce the type $br(S)$ of all binary relations on $S$: 

\begin{flagderiv}
\introduce*{}{S: *_s}{}

\step*{}{\text{Definition  }\boldsymbol{br(S)}:= S\rightarrow S\rightarrow *_p \;:\;\square}{}
\end{flagderiv}

In the rest of the article we call binary relations just relations.
The equality of relations and operations on relations are defined similarly to the set equality and set operations.

Next we define the extensional equality of relations vs the intentional equality introduced in the previous section.

\begin{flagderiv}
\introduce*{}{S: *_s}{}
\introduce*{}{R,Q:br(S)}{}
\step*{}{\text{Definition  }\subseteq(S,R,Q)\;:=(\forall x,y:S.(Rxy\Rightarrow Qxy))\;:\;*_p}{}
\step*{} {\text{Notation}: \boldsymbol{R\subseteq Q} \text{ for } \subseteq(S,R,Q)}{}
\step*{}{\text{Definition  }Ex\mhyphen eq(S,R,Q)\;:=R\subseteq Q\wedge Q\subseteq R\;:\;*_p}{}
\step*{} {\text{Notation}: \boldsymbol{R=Q} \text{ for } Ex\mhyphen eq(S,R,Q)}{\textbf{ Extensional equality}}
\end{flagderiv}

We add to the theory $\lambda D$ the following axiom of extensionality for relations. 

\begin{flagderiv}
\introduce*{}{S: *_s}{}
\introduce*{}{R,Q:br(S)}{}
\assume*{}{u:R=Q}{}
\step*{}{\boldsymbol{ext\mhyphen axiom}(S,R,Q,u)\;:=\Bot\;:\; R=_{br(S)}Q}{\textbf{Extensionality Axiom}}
\end{flagderiv}

The axiom is introduced in the last line by a primitive definition with the symbol $\Bot$ replacing a non-existing proof term.
The Extensionality Axiom states that the two types of equality are the same for binary relations. So we will use the symbol = for both and we will not elaborate on details of applying the axiom of extensionality when converting one type of equality to the other.

\subsection{Operations on Binary Relations}

Using the flag format, we introduce the identity relation $id_S$ on type $S$ and converse  $R^{-1}$ of a relation $R$.

\begin{flagderiv}
\introduce*{}{S: *_s}{}
\step*{}{\text{Definition  }id_S\;:=\lambda x,y:S.(x=_Sy)\;:\;br(S)}{\textbf{Identity relation}}

\introduce*{}{R:br(S)}{}
\step*{}{\text{Definition  }conv(S,R)\;:=\lambda x,y:S.(Ryx)\;:\;br(S)}{}
\step*{} {\text{Notation}: \boldsymbol{R^{-1}} \text{ for } conv(S,R)}{\textbf{Converse relation}}
\end{flagderiv}

Next we introduce the  operations of union $\cup$,  intersection $\cap$, and composition $\circ$ of relations.

\begin{flagderiv}
\introduce*{}{S: *_s}{}
\introduce*{}{R,Q:br(S)}{}

\step*{}{\text{Definition  }\cup(S,R,Q)\;:=\lambda x,y:S.(Rxy\vee Qxy)\;:\;br(S)}{}
\step*{} {\text{Notation}: \boldsymbol{R\cup Q} \text{ for } \cup(S,R,Q)}{\textbf{Union}}

\step*{}{\text{Definition  }\cap(S,R,Q)\;:=\lambda x,y:S.(Rxy\wedge Qxy)\;:\;br(S)}{}
\step*{} {\text{Notation}: \boldsymbol{R\cap Q} \text{ for } \cap(S,R,Q)}{\textbf{Intersection}}

\step*{}{\text{Definition  }\circ(S,R,Q)\;:=\lambda x,y:S.(\exists z:S. (Rxz\wedge Qzy))\;:\;br(S)}{}
\step*{} {\text{Notation}: \boldsymbol{R\circ Q} \text{ for } \circ(S,R,Q)}{\textbf{Composition}}
\end{flagderiv}

\subsection{Properties of Operations}

The following two technical lemmas will be used in some future proofs.

\begin{lemma}
This lemma gives a shortcut for constructing an element of a composite relation.

\begin{flagderiv}
\introduce*{}{S: *_s\;|\;R,Q:br(S)\;|\;x,y,z:S}{}
\assume*{}{u:Rxy\;|\;v:Qyz}{}
\step*{}{a:=\wedge\text{-in }(Rxy,Qyz,u,v)\;:\;Rxy\wedge Qyz}{}
\step*{}{\textbf{prod-term }(S,R,Q,x,y,z,u,v):=\exists\text{-in }(S,\lambda t.Rxt\wedge Qtz,y,a)\;:\;(R\circ Q)xz}{}
\end{flagderiv}
\label{lemma_prod-term}
\end{lemma}
\vspace{-0.5cm}

\begin{lemma}
This lemma gives a shortcut for proving equality of two relations.

\begin{flagderiv}
\introduce*{}{S: *_s\;|\;R,Q:br(S)}{}
\assume*{}{u:R\subseteq Q\;|\;v:Q\subseteq R}{}
\step*{}{\boldsymbol{rel\mhyphen equal }(S,R,Q,u,v):=\wedge\text{-in }(R\subseteq Q,Q\subseteq R,u,v)\;:\;R= Q}{}
\end{flagderiv}
\label{lemma_rel_equal}
\end{lemma}
\vspace{-0.5cm}

\begin{theorem}
For relations $R,P$ and $Q$ on $S$ the following hold.
\medskip

1) $(R^{-1})^{-1}=R.$
\medskip

2) $(R\circ Q)^{-1}=Q^{-1}\circ R^{-1}.$
\medskip

3) $(R\cap Q)^{-1}=R^{-1}\cap Q^{-1}.$
\medskip

4) $(R\cup Q)^{-1}=R^{-1}\cup Q^{-1}.$

5) $R\circ(P\cup Q)=R\circ P\cup R\circ Q.$
\medskip

6) $(P\cup Q)\circ R=P\circ R\cup Q\circ R.$
\medskip

7) $R\circ(P\cap Q)\subseteq R\circ P\cap R\circ Q.$
\medskip

8) $(P\cap Q)\circ R\subseteq P\circ R\cap Q\circ R.$
\medskip

9) $(R\circ P)\circ Q=R\circ(P\circ Q).$
\label{theorem_operations}
\end{theorem}

The formal proof is in part A of Appendix. The proof of part 2) has the form:
\begin{flagderiv}
\introduce*{}{S: *_s\;|\;R,Q:br(S)}{}
\skipsteps*{\dots}{}
\step*{}{\boldsymbol{conv\mhyphen prod}(S,R,Q):=\ldots\;:\;(R\circ Q)^{-1}=Q^{-1}\circ R^{-1}}{}
\end{flagderiv}
Its proof term $conv\mhyphen prod(S,R,Q)$ will be re-used later in the paper.
\medskip

\section{Properties of Binary Relations}
The properties of reflexivity, symmetry, antisymmetry, transitivity, and the relations of equivalence and partial order are defined in \cite{Ned14} as follows.

\begin{flagderiv}
\introduce*{}{S:*_s\,|\,R:br(S)}{}

\step*{}{\text{Definition  }\boldsymbol{refl}(S,R)\;:=\forall x:S.(Rxx)\;:\;*_p}{}

\step*{}{\text{Definition  }\boldsymbol{sym(S,R)}\;:=\forall x,y:S.(Rxy\Rightarrow Ryx)\;:\;*_p}{}

\step*{}{\text{Definition  }\boldsymbol{antisym(S,R)}\;:=\forall x,y:S.(Rxy\Rightarrow Ryx\Rightarrow x=_Sy)\;:\;*_p}{}

\step*{}{\text{Definition  }\boldsymbol{trans(S,R)}\;:=\forall x,y,z:S.(Rxy\Rightarrow Ryz\Rightarrow Rxz)\;:\;*_p}{}

\step*{}{\text{Definition  }\boldsymbol{equiv$-$relation}(S,R)\;:=refl(S,R)\wedge sym(S,R)
\wedge \;trans(S,R)\;:\;*_p}{}

\step*{}{\text{Definition  }\boldsymbol{part$-$ord}(S,R)\;:=refl(S,R)\wedge antisym(S,R)
\wedge \;trans(S,R)\;:\;*_p}{}
\end{flagderiv}

\begin{theorem}
Suppose $R$ is a relation on type $S$. Then the following hold.
\medskip

1) \textbf{Criterion of reflexivity.} 
$R$ is reflexive $\Leftrightarrow id_S\subseteq R$.
\medskip

2) \textbf{First criterion of symmetry.}
$R$ is symmetric $\Leftrightarrow R^{-1}\subseteq R$.
\medskip

3) \textbf{Second criterion of symmetry.}
$R$ is symmetric $\Leftrightarrow R^{-1}= R$.
\medskip

4) \textbf{Criterion of antisymmetry.} 
$R$ is antisymmetric $\Leftrightarrow R^{-1}\cap R\subseteq id_S$.
\medskip

5) \textbf{Criterion of transitivity.} $R$ is transitive $\Leftrightarrow R\circ R\subseteq R$.
\label{theorem:criteria}
\end{theorem}

The formal proof is in part B of Appendix. The proof of part 3) has the form:

\begin{flagderiv}
\introduce*{}{S: *_s\;|\;R:br(S)}{}
\skipsteps*{\dots}{}
\step*{}{\boldsymbol{sym\mhyphen criterion}(S,R):=\ldots\;:\;sym(S,R)\Leftrightarrow R^{-1}=R}{}
\end{flagderiv}
Its proof term $sym\mhyphen criterion(S,R)$ will be re-used later in the paper.
\medskip

\begin{theorem}
Relation $R$ on $S$ is reflexive, symmetric and antisymmetric $\Rightarrow R=id_S.$
\label{theorem:theorem_id_S}
\end{theorem}

\begin{proof}
The formal proof is in the following flag diagram.

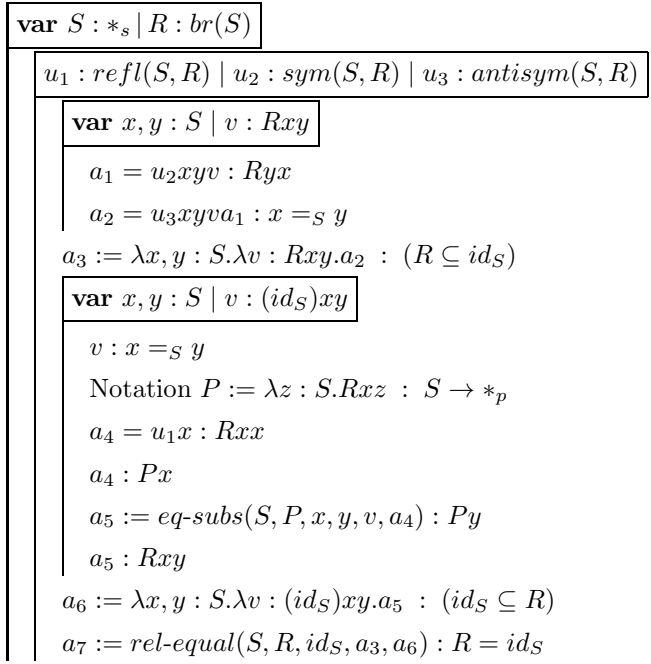
\begin{flagderiv}
\introduce*{}{S:*_s\,|\,R:br(S)}{}
\assume*{}{u_1:refl(S,R)\;|\;u_2:sym(S,R)\;|\;u_3:antisym(S,R)}{}
\introduce*{}{x,y:S\;|\;v:Rxy}{}
\step*{}{a_1=u_2xyv:Ryx}{}
\step*{}{a_2=u_3xyva_1:x=_Sy}{}
\conclude*{}{a_3:=\lambda x,y:S.\lambda v:Rxy.a_2\;:\;(R\subseteq id_S)}{}
\introduce*{}{x,y:S\;|\;v:(id_S)xy}{}
\step*{}{v:x=_Sy}{}
\step*{}{\text{Notation }P:=\lambda z:S.Rxz\;:\;S\rightarrow *_p}{}
\step*{}{a_4=u_1x:Rxx}{}
\step*{}{a_4:Px}{}
\step*{}{a_5:=eq\mhyphen subs(S,P,x,y,v,a_4):Py}{}
\step*{}{a_5:Rxy}{}
\conclude*{}{a_6:=\lambda x,y:S.\lambda v:(id_S)xy.a_5\;:\;(id_S\subseteq R)}{}
\step*{}{a_7:=rel\mhyphen equal(S,R,id_S,a_3,a_6):R=id_S}{}
\end{flagderiv}
\end{proof}

\begin{theorem}
\textbf{Invariance under converse operation.} Suppose $R$ is a relation on type $S$. Then the following hold.

1) $R$ is reflexive $\Rightarrow R^{-1}$ is reflexive.
\medskip

2) $R$ is symmetric $\Rightarrow R^{-1}$ is symmetric.
\medskip

3) $R$ is antisymmetric $\Rightarrow R^{-1}$ is antisymmetric.
\medskip

4) $R$ is transitive $\Rightarrow R^{-1}$ is transitive.
\label{theorem:converse_invariance}
\end{theorem}

\begin{proof}
1) 
\begin{flagderiv}
\introduce*{}{S:*_s\,|\,R:br(S)}{}
\assume*{}{u:refl(S,R)}{}
\introduce*{}{x:S}{}
\step*{}{ux:Rxx}{}
\step*{}{ux:R^{-1}xx}{}
\conclude*{}{a:=\lambda x:S.ux\;:\;refl(S,R^{-1})}{}
\end{flagderiv}

\newpage
2) 
\vspace{-0.2cm}
\begin{flagderiv}
\introduce*{}{S:*_s\,|\,R:br(S)}{}
\assume*{}{u:sym(S,R)}{}
\introduce*{}{x,y:S\;|\;v:R^{-1}xy}{}
\step*{}{v:Ryx}{}
\step*{}{uyx:(Ryx\Rightarrow Rxy)}{}
\step*{}{a_1:=uyxv\;:\;Rxy}{}
\step*{}{a_1:R^{-1}yx}{}
\conclude*{}{a_2:=\lambda x,y:S.\lambda v:R^{-1}xy.a_1\;:\;sym(S,R^{-1})}{}
\end{flagderiv}

3) 
\vspace{-0.2cm}

\begin{flagderiv}
\introduce*{}{S:*_s\,|\,R:br(S)}{}
\assume*{}{u:antisym(S,R)}{}
\introduce*{}{x,y:S\;|\;v:R^{-1}xy\;|\;w:R^{-1}yx}{}
\step*{}{v:Ryx}{}
\step*{}{w:Rxy}{}
\step*{}{uxy:(Rxy\Rightarrow Ryx\Rightarrow x=y)}{}
\step*{}{a_1:=uxywv\;:\;x=y}{}
\conclude*{}{a_2:=\lambda x,y:S.\lambda v:R^{-1}xy.\lambda w:R^{-1}yx.a_1\;:\;antisym(S,R^{-1})}{}
\end{flagderiv}

4)  
\vspace{-0.2cm}

\begin{flagderiv}
\introduce*{}{S:*_s\,|\,R:br(S)}{}
\assume*{}{u:trans(S,R)}{}
\introduce*{}{x,y,z:S\;|\;v:R^{-1}xy\;|\;w:R^{-1}yz}{}
\step*{}{w:Rzy}{}
\step*{}{v:Ryx}{}
\step*{}{uzyx:(Rzy\Rightarrow Ryx\Rightarrow Rzx)}{}
\step*{}{a_1:=uzyxwv\;:\;Rzx}{}
\step*{}{a_1:R^{-1}xz}{}
\conclude*{}{a_2:=\lambda x,y,z:S.\lambda v:R^{-1}xy.\lambda w:R^{-1}yz.a_1\;:\;trans(S,R^{-1})}{}
\end{flagderiv}
\end{proof}

\begin{theorem}
\textbf{Invariance under intersection.} Suppose $R$ and $Q$ are relations on type $S$. Then the following hold.

1) $R$ and $Q$ are reflexive $\Rightarrow R\cap Q$ is reflexive.

2)  $R$ and $Q$ are symmetric $\Rightarrow R\cap Q$ is symmetric.

3)  $R$ or $Q$ is antisymmetric $\Rightarrow R\cap Q$ is antisymmetric.

4)  $R$ and $Q$ are transitive $\Rightarrow R\cap Q$ is transitive.
\label{theorem:intersection_invariance}
\end{theorem}

\begin{proof}
1) 
\vspace{-0.2cm}
\begin{flagderiv}
\introduce*{}{S:*_s\,|\,R,Q:br(S)}{}
\assume*{}{u:refl(S,R)\;|\;v:refl(S,Q)}{}
\introduce*{}{x:S}{}
\step*{}{a_1:=ux\;:\;Rxx}{}
\step*{}{a_2:=vx\;:\;Qxx}{}
\step*{}{a_3:=\wedge\text{-in }(Rxx,Qxx,a_1,a_2)\;:\;(R\cap Q)xx}{}
\conclude*{}{a_4:=\lambda x:S.a_3\;:\;refl(S,R\cap Q)}{}
\end{flagderiv}

2) \vspace{-0.2cm}
\begin{flagderiv}
\introduce*{}{S:*_s\,|\,R,Q:br(S)}{}
\assume*{}{u:sym(S,R)\;|\;v:sym(S,Q)}{}
\introduce*{}{x,y:S\;|\;w:(R\cap Q)xy}{}
\step*{}{w:Rxy\wedge Qxy}{}
\step*{}{a_1:=\wedge\text{-el}_1(Rxy,Qxy,w)\;:\;Rxy}{}
\step*{}{a_2:=\wedge\text{-el}_2(Rxy,Qxy,w)\;:\;Qxy}{}
\step*{}{a_3:=uxya_1\;:\;Ryx}{}
\step*{}{a_4:=vxya_2\;:\;Qyx}{}
\step*{}{a_5:=\wedge\text{-in }(Ryx,Qyx,a_3,a_4)\;:\;(R\cap Q)yx}{}
\conclude*{}{a_6:=\lambda x,y:S.\lambda w:(R\cap Q)xy. a_5\;:\;sym(S,R\cap Q)}{}
\end{flagderiv}

3) \vspace{-0.2cm}
\begin{flagderiv}
\introduce*{}{S:*_s\,|\,R,Q:br(S)}{}
\step*{}{\text{Notation }A:=antisym(S,R):*_p}{}
\step*{}{\text{Notation }B:=antisym(S,Q):*_p}{}
\step*{}{\text{Notation }C:=antisym(S,R\cap Q):*_p}{}
\assume*{}{u:A\vee B}{}
\assume*{}{v:A}{}
\introduce*{}{x,y:S\;|\;w_1:(R\cap Q)xy\;|\;w_2:(R\cap Q)yx}{}
\step*{}{w_1:Rxy\wedge Qxy}{}
\step*{}{a_1:=\wedge\text{-el}_1(Rxy,Qxy,w_1)\;:\;Rxy}{}
\step*{}{w_2:Ryx\wedge Qyx}{}
\step*{}{a_2:=\wedge\text{-el}_1(Ryx,Qyx,w_2)\;:\;Ryx}{}
\step*{}{vxy:(Rxy\Rightarrow Ryx\Rightarrow x=y)}{}
\step*{}{a_3:=vxya_1a_2\;:\;x=y}{}
\conclude*[2]{}{a_4:=\lambda v:A.\lambda x,y:S.\lambda w_1:(R\cap Q)xy.\lambda w_2:(R\cap Q)yx.a_3\;:\;(A\Rightarrow C)}{}

\assume*{}{v:B}{}
\introduce*{}{x,y:S\;|\;w_1:(R\cap Q)xy\;|\;w_2:(R\cap Q)yx}{}
\step*{}{w_1:Rxy\wedge Qxy}{}
\step*{}{a_5:=\wedge\text{-el}_2(Rxy,Qxy,w_1)\;:\;Qxy}{}
\step*{}{w_2:Ryx\wedge Qyx}{}
\step*{}{a_6:=\wedge\text{-el}_2(Ryx,Qyx,w_2)\;:\;Qyx}{}
\step*{}{vxy:(Qxy\Rightarrow Qyx\Rightarrow x=y)}{}
\step*{}{a_7:=vxya_5a_6\;:\;x=y}{}
\conclude*[2]{}{a_8:=\lambda v:B.\lambda x,y:S.\lambda w_1:(R\cap Q)xy.\lambda w_2:(R\cap Q)yx.a_7\;:\;(B\Rightarrow C)}{}
\step*{}{a_9:=\vee\mhyphen el(A,B,C,u,a_4,a_8):C}{}
\step*{}{a_9:antisym(S,R\cap Q)}{}
\end{flagderiv}

4) \vspace{-0.3cm}
\begin{flagderiv}
\introduce*{}{S:*_s\,|\,R,Q:br(S)}{}
\assume*{}{u_1:trans(S,R)\;|\;u_2:trans(S,Q)}{}
\introduce*{}{x,y,z:S\;|\;v:(R\cap Q)xy\;|\;w:(R\cap Q)yz}{}
\step*{}{v:Rxy\wedge Qxy}{}
\step*{}{a_1:=\wedge\text{-el}_1(Rxy,Qxy,v)\;:\;Rxy}{}
\step*{}{a_2:=\wedge\text{-el}_2(Rxy,Qxy,v)\;:\;Qxy}{}
\step*{}{w:Ryz\wedge Qyz}{}
\step*{}{a_3:=\wedge\text{-el}_1(Ryz,Qyz,w)\;:\;Ryz}{}
\step*{}{a_4:=\wedge\text{-el}_2(Ryz,Qyz,w)\;:\;Qyz}{}
\step*{}{a_5:=u_1xyza_1a_3\;:\;Rxz}{}
\step*{}{a_6:=u_2xyza_2a_4\;:\;Qxz}{}
\step*{}{a_7:=\wedge\text{-in }(Rxz,Qxz,a_5,a_6)\;:\;(R\cap Q)xz}{}
\conclude*{}{a_8:=\lambda x,y,z:S.\lambda v:(R\cap Q)xy.\lambda w:(R\cap Q)yz.a_7\;:\;trans(S,R\cap Q)}{}
\end{flagderiv}
\end{proof}

\begin{theorem}
\textbf{Invariance under union.} Suppose $R$ and $Q$ are relations on type $S$. Then the following hold.
\medskip

1) $R$ or $Q$ is reflexive $\Rightarrow R\cup Q$ is reflexive.
\medskip

2)  $R$ and $Q$ are symmetric $\Rightarrow R\cup Q$ is symmetric.
\label{theorem:union_invariance}
\end{theorem}

\begin{proof}
1) \vspace{-0.3cm}
\begin{flagderiv}
\introduce*{}{S:*_s\,|\,R,Q:br(S)}{}
\assume*{}{u:refl(S,R)\;|\;x:S}{}
\step*{}{ux\;:\;Rxx}{}
\step*{}{a_1:=\vee\text{-in}_1(Rxx,Qxx,ux)\;:\;(R\cup Q)xx}{}
\step*{}{a_2:=\vee\text{-in}_2(Rxx,Qxx,ux)\;:\;(Q\cup R)xx}{}
\conclude*{}{a_3:=\lambda u:refl(S,R).\lambda x:S.a_1\;:\;(refl(S,R)\Rightarrow refl(S,R\cup Q))}{}
\step*{}{a_4(R,Q):=\lambda u:refl(S,R).\lambda x:S.a_2\;:\;(refl(S,R)\Rightarrow refl(S,Q\cup R))}{}
\step*{}{a_5:=a_4(Q,R)\;:\;(refl(S,Q)\Rightarrow refl(S,R\cup Q))}{}
\assume*{}{u:refl(S,R)\vee refl(S,Q)}{}
\step*{}{a_7:=\vee\mhyphen el(refl(S,R),refl(S,Q),refl(S,R\cup Q),u,a_3,a_5)\;:\;refl(S,R\cup Q)}{}
\end{flagderiv}

2) \vspace{-0.2cm}
\begin{flagderiv}
\introduce*{}{S:*_s\,|\,R,Q:br(S)}{}
\assume*{}{u_1:sym(S,R)\;|\;u_2:sym(S,Q)}{}
\introduce*{}{x,y:S\;|\;v:(R\cup Q)xy}{}
\step*{}{v:Rxy\vee Qxy}{}
\assume*{}{w:Rxy}{}
\step*{}{a_1:=u_1xyw\;:\;Ryx}{}
\step*{}{a_2:=\vee\text{-in}_1(Ryx,Qyx,a_1)\;:\;(R\cup Q)yx}{}
\conclude*{}{a_3:=\lambda w:Rxy. a_2\;:\;(Rxy\Rightarrow(R\cup Q)yx)}{}

\assume*{}{w:Qxy}{}
\step*{}{a_4:=u_2xyw\;:\;Qyx}{}
\step*{}{a_5:=\vee\text{-in}_2(Ryx,Qyx,a_4)\;:\;(R\cup Q)yx}{}
\conclude*{}{a_6:=\lambda w:Qxy. a_5\;:\;(Qxy\Rightarrow(R\cup Q)yx)}{}
\step*{}{a_7:=\vee\text{-el}(Rxy,Qxy,(R\cup Q)yx,v,a_3,a_6)\;:\;(R\cup Q)yx}{}
\conclude*{}{a_8:=\lambda x,y:S.\lambda v:(R\cup Q)xy. a_7\;:\;sym(S,R\cup Q)}{}
\end{flagderiv}
\end{proof}

\begin{theorem}
\textbf{Invariance under composition.} Suppose $R$ and $Q$ are relations on type $S$. Then the following hold. 
\medskip

1) $R\circ R^{-1}$ is always symmetric.
\smallskip

2) $R$ and $Q$ are reflexive $\Rightarrow R\circ Q$ is reflexive.
\smallskip

3) Suppose $R$ and $Q$ are symmetric. Then
\[R\circ Q\text{ is symmetric }\Leftrightarrow R\circ Q=Q\circ R.\]
\label{theorem:composition_invariance}
\end{theorem}

\vspace{-1cm}

\begin{proof}
1) \vspace{-0.2cm}
\begin{flagderiv}
\introduce*{}{S:*_s\,|\,R:br(S)}{}
\introduce*{}{x,y:S\,|\,u:(R\circ R^{-1})xy}{}

\step*{} {\text{Notation }P:=\lambda z:S.Rxz\wedge R^{-1}zy\;:\;S\rightarrow*_p}{}
\step*{}{u:(\exists z:S.Pz)}{}
\introduce*{}{z:S\;|\;v:Pz}{}
\step*{}{v:Rxz\wedge R^{-1}zy}{}
\step*{}{a_1:=\wedge\text{-el}_1(Rxz,R^{-1}zy,v):Rxz}{}
\step*{}{a_2:=\wedge\text{-el}_2(Rxz,R^{-1}zy,v):R^{-1}zy}{}
\step*{}{a_2:Ryz}{}
\step*{}{a_1:R^{-1}zx}{}
\step*{}{a_3:=\text{prod-term }(S,R,R^{-1},y,z,x,a_2,a_1)\;:\;(R\circ R^{-1})yx}{}
\conclude*{}{a_4:=\lambda z:S.\lambda v:Pz. a_3\;:\;(\forall z:S.(Pz\Rightarrow (R\circ R^{-1})yx))}{}
\step*{}{a_5:=\exists\text{-el }(S,P,u,(R\circ R^{-1})yx,a_4): (R\circ R^{-1})yx}{}
\conclude*{}{a_6:=\lambda x,y:S.\lambda u:(R\circ R^{-1})xy. a_5\;:\;sym(S,R\circ R^{-1})}{}
\end{flagderiv}

2)
\vspace{-0.2cm}
\begin{flagderiv}
\introduce*{}{S:*_s\;|\;R,Q:br(S)}{}
\assume*{}{u:refl(S,R)\;|\;v:refl(S,Q)}{}
\introduce*{}{x:S}{}
\step*{}{ux\;:\;Rxx}{}
\step*{}{vx\;:\;Qxx}{}
\step*{}{a_1:=\text{prod-term }(S,R,Q,x,x,x,ux,vx)\;:\;(R\circ Q)xx}{}
\conclude*{}{a_2:=\lambda x:S.a_1\;:\;refl(S,R\circ Q)}{}
\end{flagderiv}

3) Here we use the proof term $sym\mhyphen criterion(S,R)$ from Theorem \ref{theorem:criteria}.3) for the second criterion of symmetry and the proof term $conv\mhyphen prod$ from 
Theorem \ref{theorem_operations}.2).

\begin{flagderiv}
\introduce*{}{S:*_s}{}
\introduce*{}{R:br(S)}{}
\step*{}{a_1:=sym\mhyphen criterion(S,R)\;:\;sym(S,R)\Leftrightarrow (R^{-1}=R)}{}
\step*{}{a_2(R):=\wedge\text{-el}_1(sym(S,R)\Rightarrow (R^{-1}=R),(R^{-1}=R)\Rightarrow sym(S,R),a_1)\;
:\;
sym(S,R)\Rightarrow (R^{-1}=R)}{}
\step*{}{a_3(R):=\wedge\text{-el}_2(sym(S,R)\Rightarrow (R^{-1}=R),(R^{-1}=R)\Rightarrow sym(S,R),a_1)\;
:\;(R^{-1}=R)\Rightarrow sym(S,R)}{}
\done
\introduce*{}{R,Q:br(S)\;|\;u:sym(S,R)\;|\;v:sym(S,Q)}{}
\step*{}{a_4:=a_2(R)u:(R^{-1}=R)}{}
\step*{}{a_5:=a_2(Q)v:(Q^{-1}=Q)}{}
\step*{}{a_6:=conv\mhyphen prod(S,R,Q)\;:\;(R\circ Q)^{-1}=Q^{-1}\circ R^{-1}}{}
\step*{} {\text{Notation }P_1:=\lambda K:br(S).((R\circ Q)^{-1}=K\circ R^{-1})\;:\;br(S)\rightarrow *_p}{}
\step*{} {\text{Notation }P_2:=\lambda K:br(S).((R\circ Q)^{-1}=Q\circ K)\;:\;br(S)\rightarrow *_p}{}
\step*{}{a_6:P_1(Q^{-1})}{}
\step*{}{a_7:=eq\mhyphen subs(br(S),P_1, Q^{-1},Q,a_5,a_6)\;:\;
(R\circ Q)^{-1}=Q\circ R^{-1}}{}
\step*{}{a_7:P_2(R^{-1})}{}
\step*{}{a_8:=eq\mhyphen subs(br(S),P_2, R^{-1},R,a_4,a_7)\;:\;
(R\circ Q)^{-1}=Q\circ R}{}
\step*{} {\text{Notation }A:=sym(S,R\circ Q)\;:\;*_p}{}
\step*{} {\text{Notation }B:=(R\circ Q=Q\circ R)\;:\;*_p}{}

\assume*{}{w:A}{}
\step*{}{a_9:=a_2(R\circ Q)w\;:\;(R\circ Q)^{-1}=R\circ Q}{}
\step*{}{a_{10}:=eq\mhyphen sym(br(S),(R\circ Q)^{-1},R\circ Q,a_9)\;:\;R\circ Q=(R\circ Q)^{-1}}{}
\step*{}{a_{11}:=eq\mhyphen trans(br(S),R\circ Q,(R\circ Q)^{-1},Q\circ R,a_{10},a_8)\;:\;R\circ Q=Q\circ R}{}
\conclude*{}{a_{12}:=\lambda w:A.a_{11}\;:\;A\Rightarrow B}{}

\assume*{}{w:B}{}
\step*{}{w:(R\circ Q=Q\circ R)}{}
\step*{}{a_{13}:=eq\mhyphen sym(br(S),R\circ Q,Q\circ R,w)\;:\;Q\circ R=R\circ Q}{}
\step*{}{a_{14}:=eq\mhyphen trans(br(S),(R\circ Q)^{-1},Q\circ R,R\circ Q,a_8,a_{13})\;:\;(R\circ Q)^{-1}=R\circ Q}{}
\step*{}{a_{15}:=a_3(R\circ Q)a_{14}\;:\;sym(S,R\circ Q)}{}
\conclude*{}{a_{16}:=\lambda w:B.a_{15}\;:\;B\Rightarrow A}{}
\conclude*{}{a_{17}:=bi\mhyphen impl(A,B,a_{12},a_{16}):(sym(S,R\circ Q)\Leftrightarrow R\circ Q=Q\circ R)}{}
\end{flagderiv}
\end{proof}

\vspace{-0.3cm}

\section{Special Binary Relations}

\subsection{Equivalence Relation and Partition}

\begin{theorem}
\textbf{Invariance of equivalence relation under converse operation and intersection.} Suppose $R$ and $Q$ are equivalence relations on type $S$. Then the following hold.
\smallskip

1) $R^{-1}$ is an equivalence relation on $S$.
\medskip

2) $R\cap Q$ is  an equivalence relation on $S$.
\end{theorem}

\begin{proof}
1) can easily be derived from Theorem 
\ref{theorem:converse_invariance}.1), 2), 4) using intuitionistic logic.

2) can easily be derived from Theorem \ref{theorem:intersection_invariance}.1), 2), 4) using intuitionistic logic.

We skip the formal proofs.
\end{proof}

Next we formalize the fact that there is a correspondence between equivalence relations on $S$ and partitions of $S$.
Equivalence classes are introduced in \cite{Ned14} as follows.

\begin{flagderiv}
\introduce*{}{S:*_s\;|\;R:br(S)\;|\;u:equiv\mhyphen rel(S,R)}{}
\introduce*{}{x:S}{}
\step*{}{\boldsymbol{class}(S,R,u,x):=\{y:S\;|\;Rxy\}:ps(S)}{}
\step*{}{\text{Notation }\boldsymbol{[x]_R}\text{ for }class(S,R,u,x)}{}
\end{flagderiv}

Next we define a partition of type $S$:

\begin{flagderiv}
\introduce*{}{S:*_s\;|\;R:S\rightarrow ps(S)}{}
\step*{}{\boldsymbol{partition}(S,R):=(\forall x:S.x\varepsilon Rx)\wedge \forall x,y,z:S.(z\varepsilon Rx\Rightarrow z\varepsilon Ry\Rightarrow Rx=Ry))}{}
\end{flagderiv}

As usual, we can regard a partition $R$ as a collection $Rx\,( x\in S)$ of subsets of $S$. From this point of view, the above diagram expresses the standard two facts for a partition:

\begin{enumerate}
\item any element of $S$ belongs to one of subsets from the collection (namely $Rx$);
\item if the intersection of two subsets $Rx$ and $Ry$ is non-empty, then they coincide.
\end{enumerate}

(1) implies that each subset from the collection is non-empty and that the union of all subsets from the collection is $S$.

\begin{theorem}
Any equivalence relation $R$ on type $S$ is a partition of $S$ and vice versa.
\end{theorem}
\begin{proof}
The type of partitions of $S$ is $S\rightarrow ps(S)$, which is $S\rightarrow S\rightarrow *_p$, and it is the same as the type $br(S)$ of relations on $S$. The proof consists of two steps. 
\medskip

\textit{Step 1. Any equivalence relation is a partition.}

\begin{flagderiv}
\introduce*{}{S:*_s\;|\;R:S\rightarrow S\rightarrow *_p}{}
\assume*{}{u:equiv\mhyphen rel(S,R)}{}
\step*{}{a_1:=\wedge\mhyphen el_1(refl(S,R),
sym(S,R),\wedge\mhyphen el_1(refl(S,R)\wedge
sym(S,R),trans(S,R),u))
:refl(S,R)}{}
\introduce*{}{x:S}{}
\step*{}{a_2:=a_1x:Rxx}{}
\step*{}{a_2:(x\varepsilon Rx)}{}
\conclude*{}{a_3:=\lambda x:S.a_2\;:\;(\forall x:S.x\varepsilon X)}{}
\end{flagderiv}

This proves the first part of the definition of $partition(S,R)$ and the second part was proven in \cite{Ned14}, pg. 291.
\medskip

\textit{Step 2. Any partition is an equivalence relation.}

\begin{flagderiv}
\introduce*{}{S:*_s\;|\;R:S\rightarrow S\rightarrow *_p}{}
\assume*{}{u:partition(S,R)}{}
\step*{}{\text{Notation }A:=\forall x:S.(x\varepsilon Rx)}{}
\step*{}{\text{Notation }B:=\forall x,y,z:S.(z\varepsilon Rx\Rightarrow z\varepsilon Ry\Rightarrow Rx=Ry)}{}
\step*{}{u:A\wedge B}{}
\step*{}{a_1:=\wedge\mhyphen el_1(A,B,u):A}{}
\step*{}{a_2:=\wedge\mhyphen el_2(A,B,u):B}{}
\introduce*{}{x:S}{}
\step*{}{a_3:=a_1x:x\varepsilon Rx}{}
\step*{}{a_3:Rxx}{}
\conclude*{}{a_4:=\lambda x:S.a_3\;:\;refl(S,R)}{}
\introduce*{}{x,y:S\;|\;v:Rxy}{}
\step*{}{a_5:=a_1y\;:\;(y\varepsilon Ry)}{}
\step*{}{v:(y\varepsilon Rx)}{}
\step*{}{a_6:=a_2xyyva_5\;:\;Rx=Ry}{}
\step*{}{a_7:=a_1x\;:\;(x\varepsilon Rx)}{}
\step*{}{a_8:=eq\mhyphen subs(ps(S),\lambda Z:ps(S).x\varepsilon Z,Rx,Ry,a_6,a_7)\;:\;(x\varepsilon Ry)}{}
\step*{}{a_8:Ryx}{}
\conclude*{}{a_9:=\lambda x,y:S.\lambda v:Rxy.a_8\;:\;sym(S,R)}{}
\introduce*{}{x,y,z:S\;|\;v:Rxy\;|\;w:Ryz}{}
\step*{}{v:y\varepsilon Rx}{}
\step*{}{a_{10}:=a_9yzw\;:\;Rzy}{}
\step*{}{a_{10}:(y\varepsilon Rz)}{}
\step*{}{a_{11}:=a_2zxya_{10}v\;:\;Rz=Rx}{}
\step*{}{a_{12}:=a_1z\;:\;(z\varepsilon Rz)}{}
\step*{}{a_{13}:=eq\mhyphen subs(ps(S),\lambda Z:ps(S).z\varepsilon Z,Rz,Rx,a_{11},a_{12})\;:\;z\varepsilon Rx}{}
\step*{}{a_{13}:Rxz}{}
\conclude*{}{a_{14}:=\lambda x,y,z:S.\lambda v:Rxy.\lambda w:Ryz.a_{13}\;:\;trans(S,R)}{}
\step*{}{a_{15}:=\wedge\mhyphen in(refl(S,R)\wedge sym(S,R),trans(S,R),\wedge\mhyphen  in(refl(S,R),sym(S,R),a_4,a_9)
,a_{14})\;
 : equiv\mhyphen rel(S,R)}{}
\end{flagderiv}
\end{proof}

\subsection{Partial Order}

\begin{theorem}
\textbf{Invariance of partial order under converse operation and intersection.} Suppose $R$ and $Q$ are partial orders on type $S$. Then the following hold.
\medskip

1) $R^{-1}$ is a partial order on $S$.
\medskip

2) $R\cap Q$ is a partial order on $S$.
\end{theorem}
\begin{proof}
1) can easily be derived from Theorem 
\ref{theorem:converse_invariance}.1), 3), 4) using intuitionistic logic.

2) can easily be derived from Theorem \ref{theorem:intersection_invariance}.1), 3), 4) using intuitionistic logic.
We skip the formal proofs.
\end{proof}

\begin{example}
$\subseteq$ is a partial order on the power set $ps(S)$ of type $S$.
\end{example} 

\begin{proof}
This is the formal proof.
\vspace{-0.2cm}
\begin{flagderiv}
\introduce*{}{S:*_s}{}
\step*{}{\text{Notation }R:=\lambda X,Y:ps(S).X\subseteq Y\;:\;br(ps(S))}{}
\step*{}{\text{Notation }A:=refl(ps(S),R)}{}
\step*{}{\text{Notation }B:=antisym(ps(S),R)}{}
\step*{}{\text{Notation }C:=trans(ps(S),R)}{}
\introduce*{}{X:ps(S)}{}
\step*{}{a_1:=\lambda x:S.\lambda u:(x\varepsilon X).u\;:\;X\subseteq X}{}
\conclude*{}{a_2:=\lambda X:ps(S).a_1\;:\;A}{}

\introduce*{}{X,Y:ps(S)\;|\;u:X\subseteq Y\;|\;v:Y\subseteq X}{}
\step*{}{a_3:=\wedge\mhyphen in(X\subseteq Y,Y\subseteq X,u,v)\;:\;X=Y}{}
\conclude*{}{a_4:=\lambda X,Y:ps(S).\lambda u:X\subseteq Y.\lambda v:Y\subseteq X.a_3\;:\;B}{}

\introduce*{}{X,Y,Z:ps(S)\;|\;u:X\subseteq Y\;|\;v:Y\subseteq Z}{}
\introduce*{}{x:S\;|\;w:x\varepsilon X}{}
\step*{}{a_5:=uxw\;:\;(x\varepsilon Y)}{}
\step*{}{a_6:=vxa_5\;:\;(x\varepsilon Z)}{}
\conclude*{}{a_7:=\lambda x:S.\lambda w:(x\varepsilon X).a_6\;:\;X\subseteq Z}{}
\conclude*{}{a_8:=\lambda X,Y,Z:ps(S).\lambda u:X\subseteq Y.\lambda v:Y\subseteq Z.a_7\;:\;C}{}
\step*{}{a_9:=\wedge\mhyphen in(A\wedge B, C,\wedge\mhyphen in(A,B,a_2,a_4),a_8)\;:\;A\wedge B\wedge C}{}
\step*{}{a_9:part\mhyphen ord(ps(S),R)}{}
\end{flagderiv}
\end{proof}

\subsection{Well-Ordering and Transfinite Induction}

We will use the notation $\leqslant$ for partial order. In the following diagram we define the strict order $<$, the least element of a partially ordered set, and well-ordering of type $S$.

\begin{flagderiv}
\introduce*{}{S:*_s\;|\;\leqslant:br(S)\;|\;u:part\mhyphen ord(S,\leqslant)}{}
\step*{}{\text{Definition }<\,:=\lambda x,y:S.(x\leqslant y\wedge \neg (x=y))}{}
\introduce*{}{X:ps(S)\;|\;x:S}{}
\step*{}{\text{Definition }\boldsymbol{least}(S,\leqslant,X,x):=x\varepsilon X\wedge \forall y:S.(y\varepsilon X\Rightarrow x\leqslant y)}{}
\done
\step*{}{\text{Definition }\boldsymbol{well\mhyphen ord}(S,\leqslant):=part\mhyphen ord(S,\leqslant)
\wedge \forall X:ps(S).[\exists x:S.x\varepsilon X\Rightarrow \exists x:S.least(S,\leqslant,X,x)]}{}
\end{flagderiv}

\begin{theorem}
\textbf{Transfinite Induction.} Suppose $\leqslant$ is a well-ordering of type $S$. Then for any predicate $P$ on $S$:
\[\forall x:S.[(\forall y:S.(y<x\Rightarrow Py)\Rightarrow Px]\Rightarrow \forall x:S.Px.\]
\end{theorem}

\begin{proof}
Here is the formal proof.

\begin{flagderiv}
\introduce*{}{S:*_s\;|\;\leqslant:br(S)\;|\;u_1:well\mhyphen ord(S,\leqslant)\;|\;P:S\rightarrow *_p}{}
\assume*{}{u_2\;:\;\forall x:S.[\forall y:S.(y<x\Rightarrow Py)\Rightarrow Px]}{}
\step*{}{\text{Notation }A:=part\mhyphen ord(S,\leqslant)}{}
\step*{}{\text{Notation }B:=[\forall X:ps(S).(\exists x:S.x\varepsilon X\Rightarrow \exists x:S.least(S,\leqslant,X,x))]}{}
\step*{}{u_1:A\wedge B}{}
\step*{}{a_1:=\wedge\mhyphen el_1(A,B,u_1)\;:A\;}{}
\step*{}{a_2:=\wedge\mhyphen el_2(A,B,u_1)\;:B\;}{}
\step*{}{a_3:=\wedge\mhyphen el_1(refl(S,\leqslant)\wedge antisym(S,\leqslant),trans(S,\leqslant),a_1)
\;:\;refl(S,\leqslant)\wedge antisym(S,\leqslant)}{}
\step*{}{a_4:=\wedge\mhyphen el_2(refl(S,\leqslant),antisym(S,\leqslant),a_3)
\;:\;antisym(S,\leqslant)}{}
\step*{}{\text{Notation }X:=\lambda x:S.\neg Px\;:\;ps(S)}{}
\assume*{}{v_1:(\exists x:S.x\varepsilon X)}{}
\step*{}{a_5:=a_2Xv_1\;:\;[\exists x:S.least(S,\leqslant,X,x)]}{}
\introduce*{}{x:S\;|\;v_2:least(S,\leqslant,X,x)}{}
\step*{}{a_6:=\wedge\mhyphen el_1(x\varepsilon X,\forall y:S.(y\varepsilon X\Rightarrow x\leqslant y),v_2)\;:\;x\varepsilon X}{}
\step*{}{a_6:\neg Px}{}
\step*{}{a_7:=\wedge\mhyphen el_2(x\varepsilon X,\forall y:S.(y\varepsilon X\Rightarrow x\leqslant y),v_2)\;:\;[\forall y:S.(y\varepsilon X\Rightarrow x\leqslant y)]}{}
\introduce*{}{y:S\;|\;w_1:y<x}{}
\step*{}{a_8:=\wedge\mhyphen el_1(y\leqslant x,\neg (x=y),w_1)\;:\;y\leqslant x}{}
\step*{}{a_9:=\wedge\mhyphen el_2(y\leqslant x,\neg (x=y),w_1)\;:\;\neg (x=y)}{}
\assume*{}{w_2:\neg Py}{}
\step*{}{w_2:y\varepsilon X}{}
\step*{}{a_{10}:=a_7yw_2:x\leqslant y}{}
\step*{}{a_{11}:=a_4xya_{10}a_8:x=y}{}
\step*{}{a_{12}:=a_9a_{11}\;:\;\bot}{}
\conclude*{}{a_{13}:=\lambda w_2:\neg Py.a_{12}\;:\;\neg\neg Py}{}
\step*{}{a_{14}:=doub\mhyphen neg(Py)a_{13}\;:\;Py}{}
\conclude*{}{a_{15}:=\lambda y:S.\lambda w_1:y<x.a_{14}\;:\;[\forall y:S.(y<x\Rightarrow Py)]}{}
\step*{}{a_{16}:=u_2xa_{15}\;:\;Px}{}
\step*{}{a_{17}:=a_6a_{16}\;:\;\bot}{}
\conclude*{}{a_{18}:=\lambda x:S.\lambda v_2:least(S,\leqslant, X,x).a_{17}\;:\;[\forall x:S.(least(S,\leqslant, X,x)\Rightarrow \bot)]}{}
\step*{}{a_{19}:=\exists\mhyphen el(S,\lambda x:S.least(S,\leqslant, X,x),a_5,\bot,a_{18})\;:\;\bot}{}
\conclude*{}{a_{20}:=\lambda v_1:(\exists x:S.x\varepsilon X).a_{19}\;:\;\neg (\exists x:S.x\varepsilon X)}{}
\introduce*{}{x:S}{}
\assume*{}{w:\neg Px}{}
\step*{}{w:x\varepsilon X}{}
\step*{}{a_{21}:=\exists\mhyphen in(S,\lambda z:S.z\varepsilon X,x,w)\;:\;(\exists z:S.z\varepsilon X)}{}
\step*{}{a_{22}:=a_{20}a_{21}\;:\;\bot}{}
\conclude*{}{a_{23}:=\lambda w:\neg Px.a_{22}\;:\;\neg\neg Px}{}
\step*{}{a_{24}:=doub\mhyphen neg(Px)a_{23}\;:\;Px}{}
\conclude*{}{a_{25}:=\lambda x:S.a_{24}\;:\;(\forall x:S.Px)}{}
\end{flagderiv}

Here we used (twice) the Double Negation theorem with the proof term $doub\mhyphen neg$. This is the only place in this paper where we use the classical (not intuitionistic) logic.
\end{proof}

\section{Conclusion}

Starting with the definitions from \cite{Ned14} of binary relations and properties of reflexivity, symmetry, antisymmetry, and transitivity, we formalize in the theory $\lambda D$ (the Calculus of Constructions with Definitions) criteria for these properties and prove their invariance under operations of union, intersection, composition, and taking converse. We provide a formal definition of partition and formally prove correspondence  between equivalence relations and partitions. We derive a formal proof that $\subseteq$ is a partial order on power set. Finally we formally prove the principle of transfinite inductions for a type with well-ordering.

The results can be transferred to the proof assistants that are based on 
the Calculus of Constructions. Since binary relations are the abstract concepts used in many areas of mathematics, the results can be useful for further formalizations of mathematics in $\lambda D$. Our next direction of research is formalization of parts of probability theory in $\lambda D$ that we outlined in \cite{Kach18}.

\newpage
\appendix
\section*{APPENDIX}
\section{Proof of Theorem \ref{theorem_operations}}

\begin{proof}
1) 
\vspace{-0.2cm}
\begin{flagderiv}
\introduce*{}{S: *_s\;|\;R:br(S)}{}
\introduce*{}{x,y:S}{}
\assume*{}{u:(R^{-1})^{-1}xy}{}
\step*{}{u:R^{-1}yx}{}
\step*{}{u:Rxy}{}
\conclude*{}{a_1:=\lambda u:(R^{-1})^{-1}xy.u\;:\;(R^{-1})^{-1}xy\Rightarrow Rxy}{}

\assume*{}{u:Rxy}{}
\step*{}{u:R^{-1}yx}{}
\step*{}{u:(R^{-1})^{-1}xy}{}
\conclude*{}{a_2:=\lambda u:Rxy.u\;:\;Rxy\Rightarrow (R^{-1})^{-1}xy}{}

\conclude *{}{a_3:=\lambda x,y:S.a_1\;:\;(R^{-1})^{-1}\subseteq R}{}
\step*{}{a_4:=\lambda x,y:S.a_2\;:\;R\subseteq (R^{-1})^{-1}}{}
\step*{}{a_5:=rel\mhyphen equal \big(R^{-1})^{-1}, R,a_3,a_4\big)\;:\;(R^{-1})^{-1}=R}{}
\end{flagderiv}

2) \vspace{-0.2cm}
\begin{flagderiv}
\introduce*{}{S: *_s\;|\;R,Q:br(S)}{}
\step*{} {\text{Notation }A:=(R\circ Q)^{-1}\;:br(S)\;}{}
\step*{} {\text{Notation }B:=Q^{-1}\circ R^{-1}\;:br(S)\;}{}
\introduce*{}{x,y:S}{}
\step*{} {\text{Notation }P_1:=\lambda z:S.Ryz\wedge Qzx\;:\;S\rightarrow*_p}{}
\step*{} {\text{Notation }P_2:=\lambda z:S.Q^{-1}xz\wedge R^{-1}zy\;:\;S\rightarrow*_p}{}
\assume*{}{u:Axy}{}
\step*{}{u:(R\circ Q)yx}{}
\step*{}{u:(\exists z:S.P_1z)}{}
\assume*{}{z:S\;|\;v:P_1z}{}
\step*{}{v:Ryz\wedge Qzx}{}
\step*{}{a_1:=\wedge\text{-el}_1(Ryz,Qzx,v):Ryz}{}
\step*{}{a_2:=\wedge\text{-el}_2(Ryz,Qzx,v):Qzx}{}
\step*{}{a_1:R^{-1}zy}{}
\step*{}{a_2:Q^{-1}xz}{}
\step*{}{a_3:=\text{prod-term }(S,Q^{-1},R^{-1},x,z,y,a_2,a_1)\;:\;(Q^{-1}\circ R^{-1})xy}{}
\step*{}{a_3:Bxy}{}
\conclude*{}{a_4:=\lambda z:S.\lambda v:P_1z.a_3\;:\;(\forall z:S.(P_1z \Rightarrow Bxy))}{}
\step*{}{a_5:=\exists\text{-el }(S,P_1,u,Bxy,a_4):Bxy}{}
\conclude*[2]{}{a_6:=\lambda x,y:S. 
\lambda u:Axy.a_5\;:\;A\subseteq B}{}

\introduce*{}{x,y:S\;|\;u:Bxy}{}
\step*{}{u:(\exists z:S.P_2z)}{}
\assume*{}{z:S\;|\;v:P_2z}{}
\step*{}{v:Q^{-1}xz\wedge R^{-1}zy}{}
\step*{}{a_7:=\wedge\text{-el}_1(Q^{-1}xz,R^{-1}zy,v):Q^{-1}xz}{}
\step*{}{a_8:=\wedge\text{-el}_2(Q^{-1}xz,R^{-1}zy,v):R^{-1}zy}{}
\step*{}{a_7:Qzx}{}
\step*{}{a_8:Ryz}{}
\step*{}{a_9:=\text{prod-term }(S,R,Q,y,z,x,a_8,a_7)\;:\;(R\circ Q)yx}{}
\step*{}{a_{9}:(R\circ Q)^{-1}xy}{}
\step*{}{a_{9}:Axy}{}
\conclude*{}{a_{10}:=\lambda z:S.\lambda v:P_2z.a_{9}\;:\;(\forall z:S.(P_2z \Rightarrow Axy))}{}
\step*{}{a_{11}:=\exists\text{-el }(S,P_2,u,Axy,a_{10}):Axy}{}
\conclude*{}{a_{12}:=\lambda x,y:S. 
\lambda u:Bxy.a_{11}\;:\;B\subseteq A}{}
\step*{}{\boldsymbol{conv\mhyphen prod}(S,R,Q):=rel\mhyphen equal (A,B,a_6,a_{12})\;:\;(R\circ Q)^{-1}=Q^{-1}\circ R^{-1}}{}
\end{flagderiv}

3) \vspace{-0.2cm}

\begin{flagderiv}
\introduce*{}{S: *_s\;|\;R,Q:br(S)}{}
\step*{} {\text{Notation }A:=(R\cap Q)^{-1}\;:br(S)\;}{}
\step*{} {\text{Notation }B:=R^{-1}\cap Q^{-1}\;:br(S)\;}{}
\introduce*{}{x,y:S\;|\;u:Axy}{}
\step*{}{u:(R\cap Q)^{-1}xy}{}
\step*{}{u:(R\cap Q)yx}{}
\step*{} {u:Ryx\wedge Qyx}{}
\step*{}{a_1:=\wedge\text{-el}_1(Ryx,Qyx,v):Ryx}{}
\step*{}{a_2:=\wedge\text{-el}_2(Ryx,Qyx,v):Qyx}{}
\step*{}{a_1:R^{-1}xy}{}
\step*{}{a_2:Q^{-1}xy}{}
\step*{}{a_3:=\wedge\text{-in}(R^{-1}xy,Q^{-1}xy,a_1,a_2):Bxy}{}
\conclude*{}{a_4:=\lambda x,y:S.\lambda u:Axy.a_3\;:\;A\subseteq B}{}

\introduce*{}{x,y:S\;|\;u:Bxy}{}
\step*{} {u:R^{-1}xy\wedge Q^{-1}xy}{}
\step*{}{a_5:=\wedge\text{-el}_1(R^{-1}xy,Q^{-1}xy,v):R^{-1}xy}{}
\step*{}{a_6:=\wedge\text{-el}_2(R^{-1}xy,Q^{-1}xy,v):Q^{-1}xy}{}
\step*{}{a_5:Ryx}{}
\step*{}{a_6:Qyx}{}
\step*{}{a_7:=\wedge\text{-in}(Ryx,Qyx,a_5,a_6):(R\cap Q)yx}{}
\step*{}{a_7:(R\cap Q)^{-1}xy}{}
\step*{}{a_7:Axy}{}
\conclude*{}{a_8:=\lambda x,y:S.\lambda u:Bxy.a_7\;:\;B\subseteq A}{}
\step*{}{a_9:=rel\mhyphen equal (A,B,a_4,a_8):(R\cap Q)^{-1}= R^{-1}\cap Q^{-1}}{}
\end{flagderiv}

4)  \vspace{-0.2cm}
\begin{flagderiv}
\introduce*{}{S: *_s\;|\;R,Q:br(S)}{}
\step*{} {\text{Notation }A:=(R\cup Q)^{-1}\;:br(S)\;}{}
\step*{} {\text{Notation }B:=R^{-1}\cup Q^{-1}\;:br(S)\;}{}
\introduce*{}{x,y:S\;|\;u:Axy}{}
\step*{} {u:(R\cup Q)yx}{}
\step*{} {u:Ryx\vee Qyx}{}
\assume*{}{v:Ryx}{}
\step*{}{v:R^{-1}xy}{}
\step*{}{a_1:=\vee\text{-in}_1(R^{-1}xy,Q^{-1}xy,v):Bxy}{}
\conclude*{}{a_2:=\lambda v:Ryx.a_1\;:\;Ryx\Rightarrow Bxy}{}

\assume*{}{v:Qyx}{}
\step*{}{v:Q^{-1}xy}{}
\step*{}{a_3:=\vee\text{-in}_2(R^{-1}xy,Q^{-1}xy,v):Bxy}{}
\conclude*{}{a_4:=\lambda v:Qyx.a_3\;:\;Qyx\Rightarrow Bxy}{}

\step*{}{a_5:=\vee\text{-el }(Ryx,Qyx,Bxy,u,a_2,a_4)\;:\;Bxy}{}
\conclude*{}{a_6:=\lambda x,y:S.\lambda u:Axy.a_5\;:\;A\subseteq B}{}

\introduce*{}{x,y:S\;|\;u:Bxy}{}
\step*{} {u:R^{-1}xy\vee Q^{-1}xy}{}
\assume*{}{v:R^{-1}xy}{}
\step*{}{v:Ryx}{}
\step*{}{a_7:=\vee\text{-in}_1(Ryx,Qyx,v):Ryx\vee Qyx}{}
\step*{}{a_7:(R\cup Q)^{-1}xy}{}
\step*{}{a_7:Axy}{}
\conclude*{}{a_8:=\lambda v:R^{-1}xy.a_7\;:\;R^{-1}xy\Rightarrow Axy}{}

\assume*{}{v:Q^{-1}xy}{}
\step*{}{v:Qyx}{}
\step*{}{a_9:=\vee\text{-in}_2(Ryx,Qyx,v):Ryx\vee Qyx}{}
\step*{}{a_9:(R\cup Q)^{-1}xy}{}
\step*{}{a_9:Axy}{}
\conclude*{}{a_{10}:=\lambda v:Q^{-1}xy.a_9\;:\;Q^{-1}xy\Rightarrow Axy}{}

\step*{}{a_{11}:=\vee\text{-el }(R^{-1}xy,Q^{-1}xy,Axy,u,a_8,a_{10})\;:\;Axy}{}
\conclude*{}{a_{12}:=\lambda x,y:S.\lambda u:Bxy.a_{11}\;:\;B\subseteq A}{}
\step*{}{a_{13}:=rel\mhyphen equal (A,B,a_6,a_{12})\;:\;(R\cup Q)^{-1}= R^{-1}\cup Q^{-1}}{}
\end{flagderiv}

5) \vspace{-0.2cm}
\begin{flagderiv}
\introduce*{}{S: *_s\;|\;R,P,Q:br(S)}{}
\step*{} {\text{Notation }A:=R\circ(P\cup Q)\;:br(S)\;}{}
\step*{} {\text{Notation }B:=R\circ P\cup R\circ Q\;:br(S)\;}{}
\introduce*{}{x,y:S}{}
\step*{} {\text{Notation }P_0:=\lambda z:S.Rxz\wedge(P\cup Q)zy\;:\;S\rightarrow*_p}{}
\assume*{}{u:Axy}{}
\step*{}{u:(\exists z:S.P_0z)}{}
\assume*{}{z:S\;|\;v:P_0z}{}
\step*{}{v:Rxz\wedge (P\cup Q)zy}{}
\step*{}{a_1:=\wedge\mhyphen  el_1(Rxz,(P\cup Q)zy,v):Rxz}{}
\step*{}{a_2:=\wedge\mhyphen el_2(Rxz,(P\cup Q)zy,v):(P\cup Q)zy}{}
\step*{}{a_2:Pzy\vee Qzy}{}
\assume*{}{w:Pzy}{}
\step*{}{a_3:=\text{prod-term }(S,R,P,x,z,y,a_1,w)\;:\;(R\circ P)xy}{}
\step*{}{a_4:=\vee\mhyphen in_1((R\circ P)xy,(R\circ Q)xy,a_3):Bxy}{}
\conclude*{}{a_5:=\lambda w:Pzy.a_4\;:\;Pzy\Rightarrow Bxy}{}
\assume*{}{w:Qzy}{}
\step*{}{a_6:=\text{prod-term }(S,R,Q,x,z,y,a_1,w)\;:\;(R\circ Q)xy}{}
\step*{}{a_7:=\vee\text{-in }_2((R\circ P)xy,(R\circ Q)xy,a_6):Bxy}{}
\conclude*{}{a_8:=\lambda w:Qzy.a_7\;:\;Qzy\Rightarrow Bxy}{}
\step*{}{a_9:=\vee\text{-el }(Pzy,Qzy,Bxy,a_2,a_5,a_8):Bxy}{}
\conclude*{}{a_{10}:=\lambda z:S.\lambda v:P_0z.a_9\;:\;(\forall z:S.(P_0z \Rightarrow Bxy))}{}
\step*{}{a_{11}:=\exists\text{-el }(S,P_0,u,Bxy,a_{10}):Bxy}{}
\conclude*[2]{}{a_{12}:=\lambda x,y:S. 
\lambda u:Axy.a_{11}\;:\;A\subseteq B}{}

\introduce*{}{x,y:S}{}
\step*{} {\text{Notation }P_1:=\lambda z:S.Rxz\wedge Pzy\;:\;S\rightarrow*_p}{}
\step*{} {\text{Notation }P_2:=\lambda z:S.Rxz\wedge Qzy\;:\;S\rightarrow*_p}{}
\assume*{}{u:Bxy}{}
\step*{}{u:(R\circ P)xy\vee (R\circ Q)xy}{}
\assume*{}{v:(R\circ P)xy}{}
\step*{}{v:(\exists z:S.P_1z)}{}
\assume*{}{z:S\;|\;w:P_1z}{}
\step*{}{w:Rxz\wedge Pzy}{}
\step*{}{a_{13}:=\wedge\mhyphen\text{el}_1(Rxz,Pzy,w)\;:\;Rxz}{}
\step*{}{a_{14}:=\wedge\mhyphen\text{el}_2(Rxz,Pzy,w)\;:\;Pzy}{}
\step*{}{a_{15}:=\vee\mhyphen in_1(Pzy,Qzy,a_{14}):(P\cup Q)zy}{}
\step*{}{a_{16}:=\text{prod-term }(S,R,(P\cup Q),x,z,y,a_{13},a_{15})\;:\;Axy}{}
\conclude*{}{a_{17}:=\lambda z:S.\lambda
w:P_1z.a_{16}\;:\;(\forall z:S.(P_1z\Rightarrow Axy))}{}
\step*{}{a_{18}:=\exists\text{-el }(S,P_1,v,Axy,a_{17}):Axy}{}
\conclude*{}{a_{19}:=\lambda v:(R\circ P)xy.a_{18}\;:\;((R\circ P)xy\Rightarrow Axy)}{}

\assume*{}{v:(R\circ Q)xy}{}
\step*{}{v:(\exists z:S.P_2z)}{}
\assume*{}{z:S\;|\;w:P_2z}{}
\step*{}{a_{20}:=\wedge\text{-el}_1(Rxz,Qzy,w)\;:\;Rxz}{}
\step*{}{a_{21}:=\wedge\text{-el}_2(Rxz,Qzy,w)\;:\;Qzy}{}
\step*{}{a_{22}:=\vee\mhyphen in_2(Pzy,Qzy,a_{21}):(P\cup Q)zy}{}
\step*{}{a_{23}:=\text{prod-term }(S,R,(P\cup Q),x,z,y,a_{20},a_{22})\;:\;Axy}{}
\conclude*{}{a_{24}:=\lambda z:S.\lambda
w:P_2z.a_{23}\;:\;(\forall z:S.(P_2z\Rightarrow Axy))}{}
\step*{}{a_{25}:=\exists\text{-el }(S,P_2,v,Axy,a_{24}):Axy}{}
\conclude*{}{a_{26}:=\lambda v:(R\circ Q)xy.a_{25}\;:\;((R\circ Q)xy\Rightarrow Axy)}{}
\step*{}{a_{27}:=\vee\text{-el}((R\circ P)xy,(R\circ Q)xy, Axy,u,a_{19}, a_{26})\;:\;Axy}{}
\conclude*[2]{}{a_{28}:=\lambda x,y:S.\lambda
u:Bxy.a_{27}\;:\;B\subseteq A}{}
\step*{}{a_{29}:=rel\mhyphen equal (A,B,a_{12},a_{28})\;:\;R\circ(P\cup Q)=R\circ P\cup R\circ Q}{}
\end{flagderiv}

6) is proven similarly to 5).

\newpage
7) \vspace{-0.2cm}
\begin{flagderiv}
\introduce*{}{S: *_s\;|\;R,P,Q:br(S)}{}
\step*{} {\text{Notation }A:=R\circ(P\cap Q)\;:br(S)\;}{}
\step*{} {\text{Notation }B:=R\circ P\cap R\circ Q\;:br(S)\;}{}
\introduce*{}{x,y:S}{}
\step*{} {\text{Notation }P:=\lambda z:S.Rxz\wedge(P\cap Q)zy\;:\;*_p}{}
\assume*{}{u:Axy}{}
\step*{} {u:(\exists z:S.Pz)}{}
\introduce*{}{z:S\;|\;v:Pz}{}
\step*{}{v:Rxz\wedge (P\cap Q)zy}{}
\step*{}{a_1:=\wedge\text{-el}_1(Rxz,(P\cap Q)zy,v):Rxz}{}
\step*{}{a_2:=\wedge\text{-el}_2(Rxz,(P\cap Q)zy,v):(P\cap Q)zy}{}
\step*{}{a_2:Pzy\wedge Qzy}{}
\step*{}{a_3:=\wedge\text{-el}_1(Pzy,Qzy,a_2):Pzy}{}
\step*{}{a_4:=\wedge\text{-el}_2(Pzy,Qzy,a_2):Qzy}{}
\step*{}{a_5:=\text{prod-term }(S,R,P,x,z,y,a_1,a_3)\;:\;(R\circ P)xy}{}
\step*{}{a_6:=\text{prod-term }(S,R,Q,x,z,y,a_1,a_4)\;:\;(R\circ Q)xy}{}
\step*{}{a_7:=\wedge\text{-in }((R\circ P)xy,(R\circ Q)xy,a_5,a_6)\;:\;Bxy}{}
\conclude*{}{a_8:=\lambda z:S.\lambda
v:Pz.a_7\;:\;(\forall z:S.(Pz\Rightarrow Bxy))}{}
\step*{}{a_9:=\exists\text{-el }(S,P,u,Bxy,a_8):Bxy}{}
\conclude*[2]{}{a_{10}:=\lambda x,y:S.\lambda u:Axy.a_9\;:\;R\circ(P\cap Q)\subseteq R\circ P\cap R\circ Q}{}
\end{flagderiv}

8) is proven similarly to 7).

9) \vspace{-0.2cm}
\begin{flagderiv}
\introduce*{}{S: *_s\;|\;R,P,Q:br(S)}{}
\step*{} {\text{Notation }A:=(R\circ P)\circ Q\;:br(S)\;}{}
\step*{} {\text{Notation }B:=R\circ (P\circ Q)\;:br(S)\;}{}
\introduce*{}{x,y:S}{}
\step*{} {\text{Notation }P_1(x,y):=\lambda z:S.(R\circ P)xz\wedge Qzy\;:\;S\rightarrow*_p}{}
\step*{} {\text{Notation }P_2(x,y):=\lambda z:S.Rxz\wedge (P\circ Q)zy\;:\;S\rightarrow*_p}{}
\step*{} {\text{Notation }P_3(x,y):=\lambda z:S.Rxz\wedge Pzy\;:\;S\rightarrow*_p}{}
\step*{} {\text{Notation }P_4(x,y):=\lambda z:S.Pxz\wedge Qzy\;:\;S\rightarrow*_p}{}
\done
\introduce*{}{x,y:S\;|\;u:Axy}{}
\step*{} {u:(\exists z:S.P_1(x,y)z)}{}
\introduce*{}{z:S\;|\;v:P_1(x,y)z}{}
\step*{}{a_1:=\wedge\text{-el}_1((R\circ P)xz,Qzy,v):(R\circ P)xz}{}
\step*{}{a_2:=\wedge\text{-el}_2((R\circ P)xz,Qzy,v):Qzy}{}
\step*{}{a_1:(\exists z_1:S.P_3(x,z)z_1)}{}

\introduce*{}{z_1:S\;|\;w:P_3(x,z)z_1}{}
\step*{}{w:Rxz_1\wedge Pz_1z}{}
\step*{}{a_3:=\wedge\text{-el}_1(Rxz_1,Pz_1z,w):Rxz_1}{}
\step*{}{a_4:=\wedge\text{-el}_2(Rxz_1,Pz_1z,w):Pz_1z}{}
\step*{}{a_5:=\text{prod-term }(S,P,Q,z_1,z,y,a_4,a_2)\;:\;(P\circ Q)z_1y}{}
\step*{}{a_6:=\text{prod-term }(S,R,(P\circ Q),x,z_1,y,a_3,a_5)\;:\;Bxy}{}
\conclude*{}{a_7:=\lambda z_1:S.\lambda
w:P_3(x,z)z_1.a_6\;:\;(\forall z_1:S.(P_3(x,z)z_1\Rightarrow Bxy))}{}
\step*{}{a_8:=\exists\text{-el }(S,P_3(x,z),a_1,Bxy,a_7):Bxy}{}
\conclude*{}{a_9:=\lambda z:S.\lambda
v:P_1(x,y)z.a_8\;:\;(\forall z:S.(P_1(x,y)z\Rightarrow Bxy))}{}
\step*{}{a_{10}:=\exists\text{-el }(S,P_1(x,y),u,Bxy,a_9):Bxy}{}
\conclude*{}{a_{11}:=\lambda x,y:S.\lambda u:Axy.a_{10}\;:\;A\subseteq B}{}

\introduce*{}{x,y:S\;|\;u:Bxy}{}
\step*{} {u:(\exists z:S.P_2(x,y)z)}{}
\introduce*{}{z:S\;|\;v:P_2(x,y)z}{}
\step*{}{a_{12}:=\wedge\text{-el}_1(Rxz,(P\circ Q)zy,v):Rxz}{}
\step*{}{a_{13}:=\wedge\text{-el}_2(Rxz,(P\circ Q)zy,v):(P\circ Q)zy}{}
\step*{}{a_{13}:(\exists z_1:S.P_4(z,y)z_1)}{}

\introduce*{}{z_1:S\;|\;w:P_4(z,y)z_1}{}
\step*{}{w:Pzz_1\wedge Qz_1y}{}
\step*{}{a_{14}:=\wedge\text{-el}_1(Pzz_1,Qz_1y,w):Pzz_1}{}
\step*{}{a_{15}:=\wedge\text{-el}_2(Pzz_1,Qz_1y,w):Qz_1y}{}
\step*{}{a_{16}:=\text{prod-term }(S,R,P,x,z,z_1,a_{12},a_{14})\;:\;(R\circ P)xz_1}{}
\step*{}{a_{17}:=\text{prod-term }(S,R\circ P,Q,x,z_1,y,a_{16},a_{15})\;:\;Axy}{}
\conclude*{}{a_{18}:=\lambda z_1:S.\lambda
w:P_4(z,y)z_1.a_{17}\;:\;(\forall z_1:S.(P_4(z,y)z_1\Rightarrow Axy))}{}
\step*{}{a_{19}:=\exists\text{-el }(S,P_4(z,y),a_{13},Axy,a_{18}):Axy}{}
\conclude*{}{a_{20}:=\lambda z:S.\lambda
v:P_2(x,y)z.a_{19}\;:\;(\forall z:S.(P_2(x,y)z\Rightarrow Axy))}{}
\step*{}{a_{21}:=\exists\text{-el }(S,P_2(x,y),u,Axy,a_{20}):Axy}{}
\conclude*{}{a_{22}:=\lambda x,y:S.\lambda u:Bxy.a_{21}\;:\;B\subseteq A}{}
\step*{}{a_{23}:=rel\mhyphen equal(A,B,a_{11},a_{22})\;:\;(R\circ P)\circ Q=R\circ(P\circ Q)}{}
\end{flagderiv}
\end{proof}

\newpage
\section{Proof of Theorem \ref{theorem:criteria}}

\begin{proof}
Each statement here is a bi-implication, so we use the proof term \textit{bi-impl} from Lemma \ref{lemma_bi-impl}.

1) \vspace{-0.2cm}

\begin{flagderiv}
\introduce*{}{S:*_s\,|\,R:br(S)}{}
\step*{} {\text{Notation }A:=refl(S,R):*_p}{}
\step*{} {\text{Notation }B:=id_s\subseteq R:*_p}{}
\assume*{}{u:A}{}
\introduce*{}{x,y:S\;|\;v:(id_S)xy}{}
\step*{}{v:x=_Sy}{}
\step*{} {\text{Notation } P:=\lambda z:S.Rxz\;:\;S\rightarrow*_p}{}
\step*{}{ux:Px}{}
\step*{}{a_1:=eq$-$subs(S,P,x,y,v,ux)\;:\;Py}{}
\step*{}{a_1:Rxy}{}
\conclude*{}{a_2:=\lambda x,y:S.\lambda v:(id_S)xy.a_1\;:\;(id_S\subseteq R)}{}
\step*{}{a_2:B}{}
\conclude*{}{a_3:=\lambda u:A.a_2\;:\;(A\Rightarrow B)}{}

\assume*{}{u:B}{}
\introduce*{}{x:S}{}
\step*{}{a_4:=eq$-$refl(S,x)\;:\;x=_Sx}{}
\step*{}{a_4:(id_S)xx}{}
\step*{}{uxx:(id_S)xx\Rightarrow Rxx}{}
\step*{}{a_5:=uxxa_4\;:\;Rxx}{}
\conclude*{}{a_6:=\lambda x:S.a_5\;:\;(\forall x:S.Rxx)}{}
\step*{}{a_6:A}{}
\conclude*{}{a_7:=\lambda u:B.a_6\;:\;(B\Rightarrow A)}{}
\step*{}{a_8:=bi\mhyphen impl (A,B,a_3,a_7)\;:\;refl(S,R)\Leftrightarrow id_s\subseteq R}{}
\end{flagderiv}

2) and 3) are proven together as follows.

 \vspace{-0.2cm}
\begin{flagderiv}
\introduce*{}{S:*_s\,|\,R:br(S)}{}
\step*{} {\text{Notation }A:=sym(S,R):*_p}{}
\step*{} {\text{Notation }B:=R^{-1}\subseteq R:*_p}{}
\step*{} {\text{Notation }C:=R^{-1}= R:*_p}{}
\assume*{}{u:A}{}
\introduce*{}{x,y:S\;|\;v:R^{-1}xy}{}
\step*{}{v:Ryx}{}
\step*{} {uyx:(Ryx\Rightarrow Rxy)}{}
\step*{}{a_1:=uyxv\;:\;Rxy}{}
\conclude*{}{a_2:=\lambda x,y:S.\lambda u:R^{-1}xy.a_1\;:\;(R^{-1}\subseteq R)}{}

\introduce*{}{x,y:S\;|\;v:Rxy}{}
\step*{} {uxy:(Rxy\Rightarrow Ryx)}{}
\step*{}{a_3:=uxyv\;:\;Ryx}{}
\step*{}{a_3:R^{-1}xy}{}
\conclude*{}{a_4:=\lambda x,y:S.\lambda u:Rxy.a_3\;:\;(R\subseteq R^{-1})}{}
\step*{}{a_5:=rel\mhyphen equal(S,R^{-1},R,a_2,a_4)\;:\;R^{-1}=R}{}
\conclude*{}{a_6:=\lambda u:A.a_2\;:\;A\Rightarrow B}{}
\step*{}{a_7:=\lambda u:A.a_5\;:\;A\Rightarrow C}{}

\assume*{}{u:B}{}
\introduce*{}{x,y:S\;|\;v:Rxy}{}
\step*{}{v:R^{-1}yx}{}
\step*{}{uyx:(R^{-1}yx\Rightarrow Ryx)}{}
\step*{}{a_8:=uyxv\;:\;Ryx}{}
\conclude*{}{a_9:=\lambda x,y:S.\lambda v:Rxy.a_8\;:\;sym(S,R)}{}
\conclude*{}{a_{10}:=\lambda u:B.a_8\;:\;(B\Rightarrow A)}{}

\assume*{}{u:C}{}
\step*{}{u:R^{-1}\subseteq R\wedge R\subseteq R^{-1}}{}
\step*{}{a_{11}:=\wedge\text{-el}_1(R^{-1}\subseteq R,R\subseteq R^{-1},u):R^{-1}\subseteq R}{}
\step*{}{a_{11}:B}{}
\step*{}{a_{12}:=a_{10}a_{11}\;:\;A}{}
\conclude*{}{a_{13}:=\lambda u:C.a_{12}\;:\;(C\Rightarrow A)}{}
\step*{}{a_{14}:=bi\mhyphen impl (A,B,a_6,a_{10})\;:\;sym(S,R)\Leftrightarrow R^{-1}\subseteq R}{}
\step*{}{\boldsymbol{sym\mhyphen criterion}(S,R):=bi\mhyphen impl (A,C,a_7,a_{13})\;:\;sym(S,R)\Leftrightarrow R^{-1}=R}{}
\end{flagderiv}

4) \vspace{-0.2cm}
\begin{flagderiv}
\introduce*{}{S:*_s\,|\,R:br(S)}{}
\step*{} {\text{Notation }A:=antisym(S,R):*_p}{}
\step*{} {\text{Notation }B:=R\cap R^{-1}\subseteq id_S:*_p}{}

\assume*{}{u:A}{}
\introduce*{}{x,y:S\;|\;v:(R\cap R^{-1})xy}{}
\step*{}{v:R^{-1}xy\wedge Rxy}{}
\step*{}{a_1:=\wedge\text{-el}_1(R^{-1}xy,Rxy,v):R^{-1}xy}{}
\step*{}{a_2:=\wedge\text{-el}_2(R^{-1}xy,Rxy,v):Rxy}{}
\step*{}{a_1:Ryx}{}
\step*{}{uxy:Rxy\Rightarrow Ryx\Rightarrow x=y}{}
\step*{}{a_3:=uxya_2a_1\;:\;(x=y)}{}
\step*{}{a_3:(id_S)xy}{}
\conclude*{}{a_4:=\lambda x,y:S.\lambda v:(R\cap R^{-1})xy.a_3\;:\;(R\cap R^{-1}\subseteq id_S)}{}
\step*{}{a_4:B}{}
\conclude*{}{a_5:=\lambda u:A.a_4\;:\;(A\Rightarrow B)}{}

\assume*{}{u:B}{}
\introduce*{}{x,y:S\;|\;v:Rxy\;|\;w:Ryx}{}
\step*{}{w:R^{-1}xy}{}
\step*{}{a_6:=\wedge\text{-in}_1(R^{-1}xy,Rxy,w,v):(R^{-1}\cap R)xy}{}
\step*{}{a_7:=uxya_6\;:\; (id_S)xy}{}
\step*{}{a_7:x=y}{}
\conclude*{}{a_8:=\lambda x,y:S.\lambda v:Rxy.\lambda w:Ryx. a_7\;:\;antisym(S,R)}{}
\step*{}{a_8:A}{}
\conclude*{}{a_9:=\lambda u:B.a_8\;:\;(B\Rightarrow A)}{}
\step*{}{a_{10}:=bi\mhyphen impl(A,B,a_5,a_9):(antisym(S,R)\Leftrightarrow (R\cap R^{-1}\subseteq id_S))}{}
\end{flagderiv} 
 
5) \vspace{-0.2cm}
\begin{flagderiv}
\introduce*{}{S:*_s\,|\,R:br(S)}{}
\step*{} {\text{Notation }A:=trans(S,R):*_p}{}
\step*{} {\text{Notation }B:=R\circ R\subseteq R:*_p}{}
\assume*{}{u:A}{}
\introduce*{}{x,y:S}{}
\step*{} {\text{Notation }P:=\lambda z:S.Rxz\wedge Rzy\;:\;S\rightarrow *_p}{}
\assume*{}{v:(R\circ R)xy}{}
\step*{}{v:(\exists z:S.Pz)}{}
\introduce*{}{z:S\;|\;w:Pz}{}
\step*{}{w:Rxz\wedge Rzy}{}
\step*{}{a_1:=\wedge\text{-el}_1(Rxz,Rzy,w):Rxz}{}
\step*{}{a_2:=\wedge\text{-el}_2(Rxz,Rzy,w):Rzy}{}
\step*{}{a_3:=uxzya_1a_2\;:\;Rxy}{}
\conclude*{}{a_4:=\lambda z:S.\lambda w:Pz.a_3\;:\;(\forall z:S.(Pz \Rightarrow Rxy))}{}
\step*{}{a_5:=\exists{-el }(S,P,v,Rxy,a_4)\;:\;Rxy}{}
\conclude*[2]{}{a_6:=\lambda x,y:S.\lambda v:(R\circ R)xy.a_5\;:\;(R\circ R\subseteq R)}{}
\step*{}{a_6:B}{}
\conclude*{}{a_7:=\lambda u:A.a_6\;:\;(A\Rightarrow B)}{}

\assume*{}{u:B}{}
\introduce*{}{x,y,z:S\;|\;v:Rxy\;|\;w:Ryz}{}
\step*{}{a_8:=prod\mhyphen term(S,R,R,x,y,z,v,w)\;:\;(R\circ R)xz}{}
\step*{}{a_9:=uxz\;:\;((R\circ R)xz\Rightarrow Rxz)}{}
\step*{}{a_{10}:=a_9a_8\;:\;Rxz}{}
\conclude*{}{a_{11}:=\lambda x,y,z:S.\lambda v:Rxy.\lambda w:Ryz.a_{10}\;:\;trans(S,R)}{}
\step*{}{a_{11}:A}{}
\conclude*{}{a_{12}:=\lambda u:B.a_{11}\;:\;(B\Rightarrow A)}{}
\step*{}{a_{13}:=bi\mhyphen impl(A,B,a_7,a_{12}):(trans(S,R)\Leftrightarrow (R\circ R\subseteq R))}{}
\end{flagderiv} 
\end{proof}

\bibliographystyle{plain}
\bibliography{farida}

\end{document}